\documentclass[preprint,11pt]{article}
\usepackage{amssymb}
\usepackage{amsfonts}
\usepackage{mathrsfs}
\usepackage{amssymb,amsfonts,amsmath}
\usepackage{graphicx,graphics,epstopdf}
\usepackage{url,bigints}
\graphicspath{{./figures_cv/}}
\usepackage{latexsym,float,tikz}
\usepackage[ruled,linesnumbered]{algorithm2e}
\RestyleAlgo{ruled}
\SetKwComment{Comment}{/* }{ */}
\let\oldnl\nl
\newcommand{\nonl}{\renewcommand{\nl}{\let\nl\oldnl}}
\usepackage{adjustbox}
\usepackage{esint}
\usepackage{multirow}
\usepackage{ctable}

\usepackage{subfigure}
\usepackage{times}
\usepackage{multicol,enumerate}
\usepackage{epstopdf}
\usepackage{color}
\usepackage{fullpage}
\usepackage{setspace}
\usepackage[english]{babel}
\usepackage[version=3]{mhchem}
\usepackage{amsmath, amssymb, amsthm}
\usepackage{verbatim, latexsym}
\usepackage{colortbl}
\usepackage{multirow}
\usepackage[OT2,OT1]{fontenc}
\newcommand\cyr
{
\renewcommand\rmdefault{wncyr}
\renewcommand\sfdefault{wncyss}
\renewcommand\encodingdefault{OT2}
\normalfont
\selectfont
}
\DeclareTextFontCommand{\textcyr}{\cyr}

\def\Xint#1{\mathchoice
{\XXint\displaystyle\textstyle{#1}}%
{\XXint\textstyle\scriptstyle{#1}}%
{\XXint\scriptstyle\scriptscriptstyle{#1}}%
{\XXint\scriptscriptstyle\scriptscriptstyle{#1}}%
\!\int}
\def\XXint#1#2#3{{\setbox0=\hbox{$#1{#2#3}{\int}$ }
\vcenter{\hbox{$#2#3$ }}\kern-.6\wd0}}

\def\dashint{\Xint-}
\newcommand{\roma}{\mathrm{I}}
\newcommand{\romb}{\mathrm{II}}

\newcommand{\prnt}[1]{\left( #1 \right)}

\newcommand{\norm}[1]{\left\|#1\right\|}

\newcommand{\normL}[2]{\norm{#1}_{L^2\prnt{#2}}}

\newcommand{\Cov}{C_{\mathrm{ov}}}
\newcommand{\Cw}{{\rm C}_{\mathrm{weak}}}
\newcommand{\Ce}{{\rm C}_{\mathrm{est}}}
\newcommand{\Const}[1]{{\rm C}_{\mathrm{#1}}}

\newcommand{\Co}{{\rm C}_{\mathrm{ap}}(k)}
\newcommand{\Cpoin}[1]{{\rm C}_{\mathrm{poin}}(#1)}
\newcommand{\Cpoinn}[2]{{\rm C}^{#2}_{\mathrm{poin}}(#1)}
\newcommand{\locv}[3]{{#1}^{#2}_{\mathrm{#3}}}

\newcommand{\RR}{\mathcal{R}}
\newcommand{\op}{\operatorname}
\newtheorem{theorem}{Theorem}[section]

\newtheorem{assumption}{Assumption}[section]
\newtheorem{remark}{Remark}[section]
\newtheorem{lemma}{Lemma}[section]

\newtheorem{proposition}{Proposition}[section]
\numberwithin{equation}{section}
\usepackage{soul}
\setstcolor{red}
\allowdisplaybreaks

\title{On Edge Multiscale Space based Hybrid Schwarz Preconditioner for Helmholtz Problems with Large Wavenumbers}
\author{Shubin Fu\thanks{Eastern Institute for Advanced Study, Eastern Institute of Technology, Ningbo, Zhejiang 315200, P. R. China. (\texttt{shubinfu@eias.ac.cn})} \and Shihua Gong\thanks{School of Science and Engineering, The Chinese University of Hong Kong, Shenzhen, Guangdong 518172, P. R. China ({\tt{gongshihua@cuhk.edu.cn}}) SG acknowledges the support from NSFC (No. 12201535), Guangdong Basic and Applied Basic Research Foundation (No. 2023A1515011651), and  Shenzhen Stability Science Program 2022.} \and Guanglian Li\thanks{Corresponding author. Department of Mathematics, The University of Hong Kong, Pokfulam Road, Hong Kong SAR, P.R. China. ({\tt{lotusli@maths.hku.hk}}) GL acknowledges the support from Young Scientists fund (No.: 12101520) by NSFC and Early Career
Scheme (No.: 27301921) of Hong Kong RGC.} \and Yueqi Wang\thanks{Department of Mathematics, The University of Hong Kong, Pokfulam Road, Hong Kong SAR, P.R. China ({\tt{u3007895@connect.hku.hk}}) YW is supported by Hong Kong RGC via the Hong Kong PhD Fellowship Scheme.}}
\begin{document}
\maketitle

\begin{abstract}

In this work, we develop a novel hybrid Schwarz method, termed as edge multiscale space based hybrid Schwarz (EMs-HS), for solving the Helmholtz problem with large wavenumbers. The problem is discretized using $H^1$-conforming nodal finite element methods on meshes of size $h$ decreasing faster than $k^{-1}$ such that the discretization error remains bounded as the wavenumber $k$ increases. EMs-HS consists of a one-level Schwarz preconditioner (RAS-imp) and a coarse solver in a multiplicative way. The RAS-imp preconditioner solves local problems on overlapping subdomains with impedance boundary conditions in parallel, and combines the local solutions using  partition of unity. The coarse space is an edge multiscale space proposed in {[S. Fu, G. Li, R. Craster, and S. Guenneau, \emph{Multiscale Model. Simul.}, 19(4):1684–1709,
2021]}. The key idea is to first establish a local splitting of the solution over each subdomain by a local bubble part and local Helmholtz harmonic extension part, and then to derive a global splitting by means of the partition of unity. This facilitates representing the solution as the sum of a global bubble part and a global Helmholtz harmonic extension part. 

 We prove that the EMs-HS preconditioner leads to a convergent fixed-point iteration uniformly for large wavenumbers, by rigorously analyzing the approximation properties of the coarse space to the global Helmholtz harmonic extension part and to the solution of the adjoint problem. Distinctly, the theoretical convergence analysis are valid in two extreme cases: using minimal overlapping size among subdomains (of order $h$), or using coarse spaces of optimal dimension (of magnitude $k^d$, where $d$ is the spatial dimension). We provide extensive numerical results on the sharpness of the theoretical findings and also demonstrate the method on challenging heterogeneous models.    
 \end{abstract}
 \noindent \textbf{AMS subject classifications:} 65F08, 65F10, 65N55
 
\noindent \textbf{Key words:} Helmholtz equation, large wavenumber, hybrid Schwarz preconditioner, multiscale ansatz space

\section{Introduction}

The efficient solution of the Helmholtz problem with large wavenumbers on modern multiprocessor computers is  of huge importance in many practical applications, e.g., acoustics, geophysics and electromagnetics, and has long been considered one of the most challenging problems in scientific computing. Helmholtz problems are numerically challenging because (a) the solutions are often highly oscillatory and very fine discretizations are needed to resolve them, leading to large system matrices; (b) the system matrices are highly indefinite and non-normal, and standard iterative methods are not reliable; (c) the propagative nature of the problem implies that the presence of a source at any point can have considerable effect far away, thus inhibiting the performance of parallel algorithms which rely on localization techniques. Each of these essential difficulties gets worse as the wavenumber $k$ increases. 

In recent years there have been significant developments in the theory of one-level Schwarz methods for high-frequency Helmholtz problems. Since the Galerkin projection of discrete Helmholtz problems into a local domain might be singular, one-level Schwarz methods usually require stabilization, e.g., shifted-Laplacian technique \cite{erlangga2006novel, erlangga2004class}. One-level methods with symmetrized weighted prolongation and restriction by partition of unity lead to good field of values for preconditioning absorptive Helmholtz problems  \cite{graham2020domain,gong2021domain}. There are also one-level methods using wave-type truncation techniques to construct local problems \cite{kimn2007restricted,engquist2011sweepinga}, e.g. RAS-imp preconditioners \cite{gong2022convergence,gong2023convergence} using absorbing boundary condition  and the RAS-PML preconditioners \cite{galkowski2024convergence} using perfectly matched layers, in which the local solutions are combined using partition of unity. With stable local problems, these one-level methods are power contractive. Given a fixed setting of domain decomposition with generous overlap, they can be $k$-robust even without a coarse space, but the drawback is that the subdomain problems are fairly costly and the algorithms are not scalable.

To obtain a scalable domain decomposition algorithm (i.e., weakly dependent of the number of subdomains), one usually needs a two-level method by introducing a coarse space that can capture the global information and effectively transfer it among subdomains. Two-level methods utilize a standard algorithmic strategy for scalable solutions of large systems of equations and have demonstrated great success in steady-state diffusion, elasticity and fluid problems. The classical Schwarz theory \cite{toselli2006domain} can predict the performance of two-level methods for these self-adjoint problems. One basic tool is the fictitious space lemma that requires constructing a stable decomposition for each function in the discrete solution space. Whereas the coarse function space provides a coarse approximation, the local function spaces resolve the high-frequency remainder, which should be relatively small in the $L^2$-norm compared to its $H^1$-norm. The classical Schwarz theory \cite{cai1992domain,xu1992new} can also  predict the performance of certain two-level domain decomposition algorithms for solving non-symmetric and indefinite problems when the indefinite part of the operator is a small perturbation of the elliptic part. Unfortunately, for the high-frequency Helmholtz equation, the indefinite part is amplified by the square of the wavenumber $k^2$ and the classical theory gives very pessimistic estimates  \cite{graham2017domain}.

Thus, it is natural to ask whether the theory of \cite{graham2020domain,gong2021domain,gong2022convergence,gong2023convergence,galkowski2024convergence} can be improved if a coarse space were added. For example, can the addition of a coarse space speed up the information exchange among subdomains? Or can the addition of a coarse space reduce the requirement of the overlapping size in the domain decompositions? For the conventional coarse spaces (i.e., piecewise polynomials on a coarse grid), the theory of coarse approximation for Helmholtz is inextricably linked to the theory of finite element approximation of the Helmholtz problem itself; see Remark \ref{rm:mesh-resolution} for more discussions. In particular, if the diameter of the coarse mesh is not small enough, there is no guarantee that a coarse projection exists at all, let alone any accuracy. The most well-known coarse spaces for Helmholtz problems with heterogeneous coefficients are DtN coarse space and GenEO coarse space \cite{MR4299039,MR4621835,HuLi2024,Graham-geneo-2024}, which require the overlapping size of the subdomains much larger than the fine mesh size. Hence it is crucial to construct a coarse space that scales weakly with respect to the wavenumber, while ensuring a stable coarse projection. In the context of the Helmholtz problem with large wavenumbers, ideally the coarse space should overcome the so-called pollution effect with very mild resolution condition over the coarse grid, while maintain coarse grid accuracy.
 
In this work, we propose a coarse space $V_{h,0}$ using the edge multiscale (EMs) spaces \cite{MR3939320,MR3980476,fu2023wavelet,fu2021wavelet}. EMs spaces are one class of  multiscale ansatz spaces \cite{MR2718268,MR2728702,MR3834684,MR3614010,fu2021wavelet,MR4601686}, and were originally developed for approximating the solutions of elliptic PDEs with strongly oscillating coefficients. EMs spaces arise from the global splitting of the solution by a global bubble part and a global Helmholtz harmonic extension part, realized by a local splitting in each subdomain. Since the global bubble part can be recovered locally in parallel, the EMs space is constructed to approximate the global Helmholtz harmonic extension part by solving a homogeneous Hemlholtz equation with hierarchical bases up to level $\ell$ as the Dirichlet boundary data in each subdomain and then use partition of unity functions to glue the local bases together as global multiscale bases. The dimension of $V_{h,0}$ is about $H^{-d}2^{\ell+d}$, and EMs-HS is compatible with 
any sequence of overlapping subdomains that decomposes $\Omega$ and its associated partition of unity.
Inspired by \cite{HuLi2024}, we establish in Proposition \ref{prop:wavelet-basedconv} the convergence of the proposed  EMs-HS preconditioner, i.e., the contraction property of the coarse correction for the global harmonic extension part, using the standard Schatz argument with minimal requirement on the overlapping size, under very mild conditions on the coarse mesh size ($H\lesssim k^{-1}$; see Assumption \ref{ass:resolution}) and the level parameter $\ell$ (see Proposition \ref{prop:wavelet-basedconv} and Remark \ref{rm:l}). 
This is achieved as follows. In Lemma \ref{thm:proj}, we derive the local approximation properties of the coarse space, by transferring the approximation properties over the coarse skeleton to its interior by means of the transposition method \cite{MR0350177} (see also \cite[Appendix A]{fu2021wavelet}).  
Then in Lemma \ref{thm:last}, we derive its approximation properties to the global Helmholtz harmonic extension part and the solution to the adjoint problem by means of representing the global error in terms of local errors.
Proposition \ref{prop:wavelet-basedconv} indicates that in the optimal case with uniform meshes, the proposed iterative solvers only require the dimension of the coarse spaces to be about $k^d$
in order to achieve convergence for the Helmholtz equation with large wavenumbers.

The rest of the paper is structured as follows. We describe the basics on the weak formulation and FEM approximation in Section \ref{sec:prep}. In Section \ref{sec:hybrid-2}, we present  the one-level RAS-imp preconditioner and the hybrid Schwarz preconditioner. Then in Section \ref{sec:coarse-space}, we give the construction of the coarse space, and in Section \ref{sec:converge} prove the convergence of the hybrid Schwarz preconditioner. In Section \ref{sec:num}, we present 2-d and 3-d numerical tests to complement the theoretical findings. Throughout, we define the complex-valued space $V:=H^{1}(\Omega;\mathbb{C}):={W^{1,2}(\Omega;\mathbb{C})}$, equipped with the $k$-weighted norm
\[
\norm{v}_{V}:=\sqrt{\|\nabla v\|_{L^2(\Omega)}^2+k^2\|v\|_{L^2(\Omega)}^2},
\]
and similarly define the complex-valued space $V(D):=H^{1}(D;\mathbb{C})$ also equipped with the $k$-weighted norm
$\norm{\cdot}_{1,k,D}$ for any $D\subset \Omega$.
We denote by $(\cdot,\cdot)_{D}$ and $\|\cdot\|_{L^2(D)}$ the $L^2(D;\mathbb{C})$ inner product and its induced norm for any $D\subset\Omega$, $\mathrm{Re}\{\cdot\}$, $\mathrm{Im}\{\cdot\}$ and $\bar{\cdot}$ the real part, imaginary part and conjugate of a complex number. The notation $a\lesssim b$ means that there exists constant $C>0$, independent of all parameters of interest ($h,k,H$), but may depend on the triangulation and the domain $\Omega$, such that $a\leq Cb$. We also write $a\simeq b$ if $a \lesssim b$ and $b \lesssim a$.

\section{The Helmholtz equation and its finite element approximation}\label{sec:prep}
Let $\Omega\subset \mathbb{R}^d $ ($d=2,3$) be a bounded convex polygonal/polyhedron domain, $f\in (H^1(\Omega;\mathbb{C}))^\prime$ and $g\in H^{-1/2}(\partial\Omega;\mathbb{C})$. Consider the Helmholtz equation subject to the impedance boundary condition
\begin{equation}
    \label{hel_cont}
   \left\{  \begin{split}
    -\Delta u-k^2u&=f,\quad\mbox{in }\Omega,\\
    \frac{\partial u}{\partial n}+iku&=g,\quad\mbox{on }\partial \Omega, 
     \end{split}\right.
\end{equation}
where without loss of generality, we assume the wavenumber $k\geq k_0>0$. 

The variational formulation for \eqref{hel_cont} reads: find $u\in V$ such that 
\begin{equation}
\label{var_form}
    a(u,v)=F(v), \quad\forall v\in V,
\end{equation}
where the sesquilinear form $a(\cdot,\cdot): V\times V\to \mathbb{C}$ and the conjugate linear form $F: V\to \mathbb{C}$ are defined by
\begin{align*}
  a(v_1,v_2)&:= \int_\Omega\nabla v_1\cdot\nabla\overline{v}_2\mathrm{d}x-k^2\int_\Omega v_1\overline{v}_2\;\mathrm{d}x+ ik\int_{\partial \Omega} v_1\overline{ v}_2 \;\mathrm{d}s, \quad\forall v_1,v_2\in V, \\
F(v)&:=\int_\Omega f\overline{v}\;\mathrm{d}x+\int_{\partial \Omega} g\overline{v}\;\mathrm{d}s,\quad\forall v\in V.   
\end{align*}
The following properties of the sesquilinear form $a(\cdot,\cdot)$ hold \cite[Theorem 3.2 and Corollary 3.3]{MR2812565}.
\begin{theorem}[Properties of the sesquilinear form $a(\cdot,\cdot): V\times V\to \mathbb{C}$]\label{them:sesquilinear}
The following properties hold,
\begin{itemize}
\item[1.] The sesquilinear form $a(\cdot,\cdot)$ is bounded: there exists a wavenumber $k$ independent constant $\Const{b}$ such that
    \begin{align}\label{eq:sesqui-bound}
    |a(v_1,v_2)|\leq \Const{b}\|v_1\|_{V}\|v_2\|_V, \quad\forall v_1,v_2\in V.
    \end{align}
\item[2.] The G\aa rding's inequality holds
\begin{align}\label{eq:garding}
\mathrm{Re}\{a(v,v)\}+2k^2\normL{v}{\Omega}^2\geq \|v\|_V^2,\quad\forall v\in V.
\end{align}
\end{itemize}
\end{theorem}
The existence and uniqueness of the weak solution of \eqref{var_form} can be found in \cite[Proposition 8.1.3]{melenk1995generalized}. 
The adjoint problem to \eqref{var_form} is frequently utilized in the analysis of indefinite problems satisfying the G\aa rding's inequality \eqref{eq:garding}, which is defined as given $w\in V'$, find $z\in V$ such that
\begin{align}\label{eqn:weakformDual}
a^*(z, v)=(v,w)_{\Omega}, \quad\forall v\in V,
\end{align}
where $a^*(\cdot,\cdot)$ is defined by (see, e.g., \cite{MR1742312})
$a^*(w, v)=\overline{a(v,w)}$ for all $ w,v\in V$.
To discretize problem \eqref{hel_cont}, we first decompose the domain $\Omega$. Let $\mathcal{T}_h$ be a quasi-uniform triangulation of $\Omega$ with a mesh size $h$. Let $V_h$ be the $H^1$-conforming piecewise linear finite element space
\[
V_h:=\{v\in H^{1}(\Omega): v|_{T}\in \mathcal{P}_{1}(T),\;\forall T\in \mathcal{T}_h\},
\]
where $\mathcal{P}_1(T)$ denotes the space of linear polynomials on the fine element $T\in \mathcal{T}_h$. Let $\Pi_h: V\to V_h$ be the standard nodal interpolation operator. The $H^1$-conforming Galerkin approximation to problem \eqref{var_form} is to find $u_h\in V_h$ such that 
\begin{equation}\label{eq:discrete-global}
    a(u_h,v_h)=F(v_h), \quad\forall v_h\in V_h.
\end{equation}
Throughout, we assume the approximation property for the FEM space $V_h$: for all $v\in H^2(\Omega)$, there holds
\begin{align}\label{approxi}
   \inf_{v_h\in V_h} \left(\|v-v_h\|_{L^2(\Omega)}+h\|\nabla (v- v_h)\|_{L^2(\Omega)}\right)\lesssim h^2|v|_{H^2(\Omega)}.
\end{align}
Using property \eqref{approxi}, we can establish the following stability and quasi-optimality results of the discrete problem \cite[Proposition 8.2.7]{melenk1995generalized}. 
\begin{theorem}[Discrete inf-sup condition]
\label{thm_1}
Let $\Omega$ be a convex polygon/polyhedron domain. Then there exists a constant $C_0$ independent of $k$ and $h$, such that if the triangulation mesh satisfy 
\begin{align}\label{eq:fine-mesh-cond}
hk^{2}\leq C_0,
\end{align}
then the following discrete inf-sup condition holds
\begin{equation*}
    \inf_{v_h\in V_h\backslash\{0\}}\sup_{w_h\in V_h\backslash\{0\}}\frac{ |a(v_h,w_h)|}{\|v_h\|_{V}\|w_h\|_{V}}
    \geq \frac{1}{\Const{stab}k}
\end{equation*}
with the constant $\Const{stab}$ independent of $h$ and $k$. Then \eqref{eq:discrete-global} is well-posed, and its solution $u_h$ satisfies
\begin{equation*}
    \|u-u_{h}\|_{V} \leq\Const{stab}  hk (\|f\|_{L^2(\Omega)}+\|g\|_{L^2(\partial\Omega)}).
\end{equation*}
\end{theorem}

\begin{remark}\label{rm:mesh-resolution}
The mesh resolution condition \eqref{eq:fine-mesh-cond} guarantees that the numerical solution $u_h$ achieves a high accuracy as the wavenumber $k$ increases. 
It can be relaxed if $H^1$-conforming piecewise $p$ order finite element space is used with $p$ large or mere existence of numerical solutions is pursued. The existence of numerical solutions with large wavenumber $k$ was studied in the seminal work of Ihlenburg and Babu\v{s}ka \cite{ihlenburg1995dispersion, ihlenburg1995finite}. Under the assumption that the stability of the Helmholtz problem is independent of $k$, the Galerkin method using the $h$-version of FEM is quasi-optimal when $k^{p+1}h^p$ is sufficiently small \cite{melenk1995generalized, sauter2006refined, melenk2010convergence, MR2812565}. For the $hp$-version (where accuracy is increased by decreasing $h$ and increasing $p$), the Galerkin method is quasi-optimal if $hk/p$ is sufficiently small and $p$ grows logarithmically with $k$ by \cite{melenk2010convergence, MR2812565}. Under a weaker condition that $k^{2p+1}h^{2p}$ is small enough, the numerical solution exists and the numerical error is bounded by the data uniformly with respective to $k$ \cite{du2015preasymptotic}. In all these cases  with quasi-uniform meshes, the magnitude of the dimensions of the piecewise polynomial function spaces is  much larger than $k^d$.
\end{remark}

\section{The hybrid Schwarz method}\label{sec:hybrid-2}
Now we present the basic assumption for the domain decomposition and the notation for the Schwarz preconditioners. Let 
$\mathcal{T}_H : = \{K_i~|~i=1,2,\cdots, N\}
$ be a quasi-uniform partition of the domain $\Omega$ into a structured mesh with a mesh size $H$ such that $\mathcal{T}_h$ is its subdivision, where $N$ is the number of the coarse elements (or the subdomains).  Let $\Omega_{i}$ be the subdomain associated with $K_{i}\in \mathcal{T}_H$, defined by
$\Omega_{i}=\bigcup\left\{T\in\mathcal{T}_h:\text{dist}(T,K_{i})\leq\delta_{i}\right\}$.
Denote by $\delta:=\min_{K_{i}\in \mathcal{T}_H}\delta_{i}$ the smallest value. Then $\{\Omega_{i}\}_{i=1}^N$ is an overlapping partition of $\Omega$, i.e., 
$\overline{\Omega}=\bigcup_{i=1}^N\overline{\Omega}_{i}$. Let $d_i:=\op{diam}(\Omega_i)$ be the diameter of $\Omega_i$. By construction, $d_i\leq H+\delta_i$. Throughout, we assume 
\begin{align}\label{eq:c-diameter}
d_i\approx H, \quad\forall i=1,\cdots,N.
\end{align}
We make the {\it finite-overlap assumption}: there exists a finite number $\Lambda>1$ independent of $N$ such that
\begin{align}\label{eq:finite-overlap}
\Lambda=\max\left\{\#\Lambda_{i}:K_{i}\in\mathcal{T}_H\right\}, \quad \text{ with }\Lambda_{i}=\{j:\bar{\Omega}_{i}\cap\bar{\Omega}_{j}\neq \emptyset\}.
\end{align}

The coarse scale mesh size $H$ is required to satisfy the so-called scale resolution assumption. 
\begin{assumption}[Scale Resolution Assumption]\label{ass:resolution}
The coarse mesh size $H$ is sufficiently small so that $\max_{i=1,2,\ldots,N}\{\Cpoinn{\Omega_i}{1/2}\}Hk<1$,
with the Poincar\'{e} constant $\Cpoin{\Omega_i}: =\max_{v\in H^1_0(\Omega_i)}
\frac{\|v\|_{L^2(\Omega_i)}^2}{\|H\nabla v\|_{L^2(\Omega_i)}^2}$ depending on the shape of $\Omega_i$ but not on $H$.
\end{assumption}

Throughout, we denote 
\begin{equation}
\Ce:=\Big(1-\max_{i=1,2,\ldots,N}\{\Cpoin{\Omega_i}\}(Hk)^2\Big)^{-1}.
\end{equation}
Note that for most regular domains, the Friedrichs' inequality implies that $\Cpoin{\Omega_i}:={d_i^2}/({H\pi})^2$. Hence, Assumption \ref{ass:resolution} with \eqref{eq:c-diameter} is equivalent to $H\le \pi k^{-1}$.
Next, we introduce a partition of unity $\{\chi_{i}\}_{i=1}^N$ subordinate to the cover $\{\Omega_{i}\}_{i=1}^N$, which is assumed to satisfy for some $C_{\infty}$ and $C_{\op{G}}$ independent of $(h,k,H)$,
\begin{equation}\label{eq:pum-property}
\left\{
\begin{aligned}
\sum\limits_{i=1}^N\chi_{i}&=1 \text{ in } \bar{\Omega},\quad\text{and}\quad {\text{supp}(\chi_{i})}\subset\bar{\Omega}_{i},\\
 \| \chi_{i}\|_{L^{\infty}(\bar{\Omega}_{i})}&\leq C_{\infty},\quad\text{and}\quad 
\|\nabla \chi_{i}\|_{L^{\infty}(\bar{\Omega}_{i})}\leq {\delta_{i}}^{-1}C_{\op{G}}.
\end{aligned}
\right.
\end{equation}
It follows immediately from \eqref{eq:finite-overlap} that, for $v_i\in H^1(\Omega_i)$,
\begin{equation}\label{eq:sumpou-ieq}
\left\|\sum_{i=1}^N \chi_i v_i\right\|_{L^2(\Omega)}^2 \le \Lambda \sum_{i=1}^N \left\| \chi_i v_i\right\|_{L^2(\Omega_i)}^2,\quad \text{and}\quad \left\|\sum_{i=1}^N\nabla( \chi_i v_i)\right\|_{L^2(\Omega)}^2 \le \Lambda \sum_{i=1}^N \left\| \nabla(\chi_i v_i)\right\|_{L^2(\Omega_i)}^2.
\end{equation}

\subsection{RAS-imp preconditioner}
Now we formulate the one-level Schwarz domain decomposition method in an operator form.
First we define the linear operator ${A}:V_{h}\to V_{h}'$ by
\begin{align}\label{eq:Ah-linear}
\langle {A}v_h,w_h\rangle:=a(v_h,w_h),\quad \forall v_h, w_h \in V_{h},
\end{align}
where $\langle \cdot,\cdot\rangle$ denotes the duality paring. 
Then the global linear system can be written as 
\begin{equation}
    \label{lin_form}
   A u_h=F_h,
\end{equation}
where $F_h\in V_h'$ is defined as $\langle F_h,v_h\rangle := F(v_h)$~ for all $v_h\in V_h$.

Next, we consider the local problems with impedance boundary conditions. Let the local function spaces be
$V_{h,i}:=\{v_h|_{\overline{\Omega}_{i}}: v_h\in V_h\}$ and define the local operator $A_i :V_{h,i}\to V_{h,i}'$ by 
\begin{align*}
\langle {A}_{i}u_i,v_i\rangle:=a_{i}(u_i,v_i), \forall u_i, v_i \in V_{h,i},
\end{align*}
where the local sesquilinear form $a_{i}(\cdot,\cdot)$ on the local space $V_{h,i} \times V_{h,i}$ is given by
\begin{equation*}
\label{local_ses}
    a_{i}(u_i,v_i) 
    := \int_{\Omega_i}\nabla u_i\cdot\nabla\overline{ v}_i\mathrm{d}x
    -k^2\int_{\Omega_i} u_i\overline{v}_i\;\mathrm{d}x
    + ik\int_{\partial \Omega_i} u_i\overline{ v}_i \;\mathrm{d}s,\quad
    \forall u_i, v_i\in V_{h,i}.
\end{equation*}
Then we define prolongations and restrictions to link local and global problems. Note that the functions in $V_h$ and $V_{h,i}$ are uniquely determined by their nodal values at the nodes $\{x_j: j\in \mathcal{I}\}$ in $\overline{\Omega}$ and $\{x_j: j\in \mathcal{I}_i\}$ in $\overline{\Omega}_i$, respectively, using the index sets $\mathcal{I}$ and $\mathcal{I}_i$. We define the prolongation ${R}_i^T: V_{h,i} \to V_h$  by
\begin{align*}
R_i^Tv_i(x_j)=\left\{
\begin{aligned}
v_i(x_j), \quad j\in \mathcal{I}_i,\\
0,\quad \text{otherwise}.
\end{aligned}\right.
\end{align*}
Note that the prolongation ${R}_i^T$ is defined nodewise and ${R}_i^Tv_i $ is an $H^1$-conforming finite element approximation to the zero extension of $v_{i}\in V_{h,i}$. (The zero extension is not in $H^1(\Omega)$ in general). We define $R_i: V'_{h}\to V'_{h,i}$ by duality. To make the prolongation more stable, we employ the weighted prolongation $\tilde{R}_i^T:V_{h,i}\to V_h$ by the functions of partition of unity:
\begin{equation*}
\tilde{R}_i^T v_{i}:  = {R}_i^T \Pi_{h,i}(\chi_i v_i),
\end{equation*}
where $\Pi_{h,i}$ is the nodal interpolation of $V_{h,i}$. Note that 
\begin{align}\label{eq:pou}
\sum_{i=1}^N \tilde{R}_i^T(v_h|_{\Omega_i})=v_h, \quad\forall v_h\in V_h.
\end{align}

The restricted additive Schwarz preconditioner with local impedance boundary condition (RAS-imp) is the operator $B_{\op{RAS-imp}}: V_{h}'\to V_h$ defined by
\begin{equation*}
B_{\op{RAS-imp}} = \sum_{i=1}^N \tilde{R}_i^T  A_i^{-1}R_i.
\end{equation*}
The preconditioned operator $B_{\op{RAS-imp}} A :V_h\to V_h$ can be written as
\begin{equation*}
B_{\op{RAS-imp}}A = \sum_{i=1}^N  \tilde{R}_i^T Q_i,
\end{equation*}
where the local Schwarz operator $Q_i:V_h\to V_{h,i}$ is defined by
\begin{align*}
a_i(Q_i v, w_i) := a(v,R_i^T w_i),\quad \forall v\in V_h\text{ and } w_i \in V_{h,i}.
\end{align*}
\begin{theorem}[Local-global approximation]
\label{thm:pum-2}
Let $v_i\in V_{h,i}$ for $i=1,2,\cdots,N$, and let
$v_h:=\sum_{i=1}^N \tilde{R}_i^T v_i\in V_h$.
Then there exists $\Const{I}$ independent of $(h,k,H)$ such that
\begin{align}
\|v_h\|_{L^2(\Omega)}&\leq \sqrt{\Lambda}\left(\sum_{i=1}^N C_{\infty}^2\|v_i\|_{L^2(\Omega_i)}^2
+C_{I}^2h^2\left\|\nabla (\chi_i v_i)\right\|_{L^2(\Omega_i)}^2\right)^{1/2},\nonumber\\
\|\nabla v_h\|_{L^2(\Omega)}&\leq \sqrt{\Lambda(1+\Const{I}^2)}\left(\sum_{i=1}^N
\left\|\nabla (\chi_i v_i)\right\|_{L^2(\Omega_i)}^2\right)^{1/2}.\label{eq:glo-err-pum2}
\end{align}
\end{theorem}
\begin{proof}
This follows from  the triangle inequality and the stability of  $\Pi_{h,i}$ (see, e.g., \cite{MR520174}),
\begin{align*}
\left\|(I-\Pi_{h,i})(\chi_i v_i)\right\|_{L^2(\Omega_i)}
+h\left\|\nabla (I-\Pi_{h,i})(\chi_i v_i)\right\|_{L^2(\Omega_i)}
\leq \Const{I}h\left\|\nabla (\chi_i v_i)\right\|_{L^2(\Omega_i)},\quad\forall v_i\in V_{h,i},
\end{align*}
and \eqref{eq:sumpou-ieq} by the finite-overlap assumption in \eqref{eq:finite-overlap}.
\end{proof}
\subsection{The EMs-HS preconditioner}
The hybrid preconditioner is a multiplicative assembly of the one level RAS-imp preconditioner $B_{\op{RAS-imp}}$ with a coarse solver, which is determined by a properly defined coarse space $V_{h,0}$ (see Section \ref{sec:coarse-space} for details). Assuming a natural prolongation $R_0^T:V_{h,0} \to V_h$, we define the coarse problem 
\begin{align*}
A_0 := R_0 A R_0^T : V_{h,0}\to V_{h,0}'.
\end{align*}
where $R_0 :V_{h}'\to V_{h,0}'$ is defined by duality. Then the EMs-HS preconditioner is defined by 
\begin{equation*}
B_{\op{hybrid}} : =  R_0^TA_0^{-1} R_0(I - A B_{\op{RAS-imp}}) + B_{\op{RAS-imp}}.
\end{equation*}
Note that the RAS-imp preconditioner $B_{\op{RAS-imp}}$ is applied only once in the computation of $B_{\op{hybrid}}$. More precisely, given a residual equation $Ae^{(0)} = r^{(0)}$, one first computes the one-level correction by $e^{(1)} = B_{\op{RAS-imp}}r^{(0)}$, then calculates the new residual $r^{(1)} = r^{(0)} - A e^{(1)}$ and solves for the coarse correction $e^{(2)} = R_0^TA_0^{-1} R_0r^{(1)}$. Then the final correction is  
$B_{\op{hybrid}} r^{(0)} = e^{(1)}+e^{(2)}$.

The hybridly preconditioned operator can be written as 
\begin{equation*}
\begin{aligned}
B_{\op{hybrid}}A &= R_0^TA_0^{-1} R_0A(I -  B_{\op{RAS-imp}}A) + B_{\op{RAS-imp}}A\\
\end{aligned}.
\end{equation*}
Analogously, we define the coarse Schwarz operator $Q_0 = A_0^{-1} R_0A:V_h\to V_{h,0} $. Using the hybrid preconditioner $B_{\op{hybrid}}$, we obtain a fixed-point iteration for \eqref{lin_form}
\begin{align*}
u^{(n+1)}_h = u^{(n)}_h +B_{\op{hybrid}} (F_h - A u^{(n)}_h), 
\end{align*}
with its associate error transfer operator being
\begin{equation*}
\begin{aligned}
I-B_{\op{hybrid}}A  &= \left(I-R_0^TQ_0\right)\left(I - B_{\op{RAS-imp}}A\right) 
 = \Big(I-R_0^TQ_0\Big)\left(I - \sum_{i=1}^N\tilde{R}^T_i Q_i\right).
\end{aligned}
\end{equation*}
Furthermore, \eqref{eq:pou} implies
\begin{equation*}
\left(I-B_{\op{hybrid}}A \right)v_h 
= \left(I-R_0^TQ_0\right)\sum_{i=1}^N \tilde{R}^T_i  \left(v_h|_{\Omega_i}- Q_i v_h\right) \quad\forall v_h\in V_h.
\end{equation*}
Equivalently, we derive for any $v_h\in V_h$, 
\begin{align}
B_{\op{hybrid}}A v_h=\left(I-R_0^TQ_0\right)v_h^{\partial}+v_h \quad\mbox{with } v_h^{\partial}:=\sum_{i=1}^N \tilde{R}^T_{i}\left(Q_i v_h-v_h|_{\Omega_i}\right). \label{eq:proj-edge-err}
\end{align}
The main objective is to show that the operator $I-B_{\op{hybrid}}A $ is contractive under minimum requirement on the overlapping size $\delta$ and a mild resolution condition on $H$.
Finally, we give an auxiliary result on the discrepancy between the local Schwarz projection $Q_i v_h$ and the restriction  $v_h|_{\Omega_i}$.

\begin{lemma}\label{lem:decomp}
Let the fine grid mesh size $h$, the coarse grid mesh size $H$ and the diameter $d_i$ of each subdomain $\Omega_i$ satisfy \eqref{eq:fine-mesh-cond}, Assumption \ref{ass:resolution} and \eqref{eq:c-diameter}, respectively. Given $v_h\in V_h$, let $v_h^{\partial}$ be defined in \eqref{eq:proj-edge-err}.  Then with $\sigma:=(kH)^{-3/2}(kh)^{-1/2}(k\delta)^{-1}$, there exists $C_{\partial}$ independent of $(h,k,H)$ such that
\begin{align}\label{eq:vhpar}
\|v_h^{\partial}\|_V\leq C_{\partial}\sigma \|v_h\|_V.
\end{align}
\end{lemma}
\begin{proof}
By \cite[Lemma 4.4]{HuLi2024},  for each $i=1,\cdots,N$, there exists $\sigma_i:=(kd_i)^{-3/2}(kh)^{-1/2}$ such that
\begin{align}\label{eq:local-harmonic}
\left\|Q_i v_h-v_h|_{\Omega_i}\right\|_{V(\Omega_i)}\lesssim \sigma_i\|v_h\|_{V(\widetilde{\Omega}_{i})},
\end{align}
where $\widetilde{\Omega}_{i}$ satisfies $
\overline{\widetilde{\Omega}}_{i}:=\cup\{\overline{T}\in \mathcal{T}_h: \overline{T}\cap\Omega_i\neq \emptyset\}.$
This, Theorem \ref{thm:pum-2} and \eqref{eq:proj-edge-err} lead to 
\begin{align*}
\|v_h^{\partial}\|_V^2
&\lesssim(1+\delta_i^{-2}k^{-2})\sum_{i=1}^N \left\|Q_i v_h-v_h|_{\Omega_i}\right\|_{V(\Omega_i)}^2\lesssim\sigma^2
\sum_{i=1}^N \|v_h\|_{V(\widetilde{\Omega}_{i})}^2.
\end{align*}
Finally, the finite-overlap assumption \eqref{eq:finite-overlap} ensures that $\{\widetilde{\Omega}_i\}_{i=1}^N$ satisfies the same assumption. This and \eqref{eq:fine-mesh-cond} yield the desired assertion.
\end{proof}
\section{Coarse space}\label{sec:coarse-space}
To define the coarse space, we employ the local Helmholtz harmonic space $V_i^\partial$ with Dirichlet boundary data
\begin{align}\label{eq:local-space}
V_i^{\partial}=\{v_h\in V_{h,i}: a_i(v_h,w_h)=0\; \forall w_h\in V_{h,i}^0\}.
\end{align}
Here, $V_{h,i}^0\subset V_{h,i}$ denotes the space consisting of functions vanishing on the boundary $\partial\Omega_i$, namely, 
\[
V_{h,i}^0:=\{v_h\in V_{h,i}:\;v_h=0\text{ on }\partial\Omega\}.
\]
Utilizing the weighted prolongation $\tilde{R}_i^T: V_{h,i}\to V_h$ and the local spaces $V_i^{\partial}$, we define a global space
\begin{align}\label{eq:glo-edge-space}
V^{\partial}:=\mathrm{span}\left\{\sum_{i=1}^N \tilde{R}_i^T v_i: v_i\in V_i^{\partial}\right\}.
\end{align}
Below we construct the coarse space $V_{h,0}$ following \cite{fu2021wavelet}, which approximates the global space $V^{\partial}$ \eqref{eq:glo-edge-space} and also serves as a good ansatz space for the adjoint problem \eqref{eqn:weakformDual}. 
\subsection{Hierarchical subspace splitting over $I=:[0,1]$}\label{sec:wavelets}
First, we introduce the hierarchical bases on the unit interval $I:=[0,1]$, which facilitate hierarchically splitting the space $L^2(I)$ \cite{MR1162107}. Let $\ell\in \mathbb{N}$ be the level parameter and $h_{\ell}:=2^{-\ell}$ be the mesh size, respectively. The grid points on level $\ell$ are given by
\[
x_{\ell,j}=j\times h_{\ell},\quad 0\leq j\leq 2^{\ell}.
\]
We define the bases on level $\ell$ by
\begin{equation*}
\psi_{\ell,j}(x)=
\left\{
\begin{aligned}
&1-|x/h_{\ell}-j|, &&\text{ if }  x\in [(j-1)h_{\ell},(j+1)h_{\ell}]\cap [0,1],\\
&0, &&\text{ otherwise},
\end{aligned}
\right.
\end{equation*}
and the index set on each level $\ell$ by
\begin{equation*}
B_{\ell}:=\Bigg\{
j\in\mathbb{N}\Bigg|
\begin{aligned}
&j=1,\cdots,2^{\ell}-1, j \,\rm{ is\, odd }, &&\rm{if }\,\ell>0\\
&j=0,1,&&\rm{ if }\,\ell=0
\end{aligned}
\Bigg\}.
\end{equation*}
The subspace of level $\ell$ is
$W_{\ell}:=\text{span}\{\psi_{\ell,j}: j\in B_{\ell}\}$.
We denote by $V_{\ell}$ the subspace in $L^2(I)$ up to level $\ell$, which is the direct sum of subspaces $V_{\ell}:=\oplus_{m\leq\ell}W_{m}$.
This gives a hierarchical structure of the subspace $V_{\ell}$:
$V_{0}\subset V_{1}\subset \cdots\subset V_{\ell}\subset V_{\ell+1}\cdots$.
Furthermore, the following hierarchical decomposition of the space $L^2(I)$ holds
\[
L^2(I)=\lim_{\ell\to\infty}\oplus_{m\leq\ell}W_{m}.
\]
The hierarchical decomposition of the space $L^2(I^{d-1})$ for $d=2,3$ can be defined by means of tensor product, which is denoted as $V_{\ell}^{\otimes^{d-1}}$. We will use the subspace $V_{\ell}^{\otimes^{d-1}}$ to approximate the restriction of the exact solution $u$ on the coarse skeleton $\partial\mathcal{T}^H:=\cup_{T\in \mathcal{T}^H}\partial T$.

\begin{proposition}[Approximation properties of the hierarchical space $V_{\ell}$]\label{prop:approx-wavelets}
Let $d=2,3$, $s>0$ and let $\mathcal{I}_{\ell}: C(I^{d-1})\to V_{\ell}^{\otimes^{d-1}}$ be $L^2$-projection for each level $\ell\geq 0$, then there holds,
\begin{align}\label{prop:approx-wavelets}
\|v-\mathcal{I}_{\ell}v\|_{L^2(I^{d-1})}&\lesssim 2^{-s\ell}|v|_{H^s(I^{d-1})}\quad\text{for all }v\in H^s(I^{d-1}).
\end{align}
Here, $|\cdot|_{H^s(I^{d-1})}$ denotes the Gagliardo seminorm in the fractional Sobolev space (or Slobodeskii space) $H^s(I^{d-1})$, given by
\begin{align*}
|v|^2_{H^s(I^{d-1})}:=\int_{I^{d-1}}\int_{I^{d-1}}\frac{|v(x)-v(y)|^2}{|x-y|^{d-1+2s}}\mathrm{d}x\mathrm{d}y.
\end{align*}
The corresponding full norm is denoted as $\|\cdot\|_{H^s(I^{d-1})}$.
\end{proposition}
\begin{proof}
If $s\geq 1$, then the assertion is standard, see, e.g. \cite{MR520174}, together with the interpolation method. Therefore, we only focus on the case with $s<1$.

Let $d=2$. A straightforward calculation leads to 
\begin{align}\label{eq:111222}
\|v-\mathcal{I}_{\ell}v\|_{L^2(I)}^2&=\sum_{j=0}^{2^{\ell}-1}\int_{x_{\ell,j}}^{x_{\ell,j+1}}|v-\mathcal{I}_{\ell}v|^2\mathrm{d}x.
\end{align}
Note that by definition of $\mathcal{I}_{0}$ on the reference interval $I$, $\mathcal{I}_0: H^s(I)\to V_0$ is stable, 
\begin{align*}
\|\mathcal{I}_{0}v\|_{L^2(I)}\leq \|v\|_{L^2(I)}.
\end{align*}
Consequently, for any $w_{0}\in V_{0}$, there hold
\begin{align*}
\|v-\mathcal{I}_{0}v\|_{L^2(I)}&=
\left\|(v-w_{0})-\mathcal{I}_{0}(v-w_{0})\right\|_{L^2(I)}\\
&\leq \left\|v-w_{0}\right\|_{L^2(I)}.
\end{align*}
By the poincar\'{e} inequality in the fractional Sobolev space with $s\in (0,1)$, see, e.g., \cite{MR3095149}, we obtain 
\begin{align*}
\min\limits_{c\in\mathbb{R}}\left\|v-c\right\|_{L^2(I)}
\lesssim \left|v\right|_{H^s(I)}.
\end{align*}
Combining the last two inequalities, and together with the fact that constant belongs to $V_0$, we obtain 
\begin{align*}
\|v-\mathcal{I}_{0}v\|_{L^2(I)}\lesssim \left|v\right|_{H^s(I)}.
\end{align*}
This, combining with a scaling argument and \eqref{eq:111222}, leads to the desired result for $d=2$. The proof for $d=3$ can be proceeded in a similar manner. 
\end{proof}

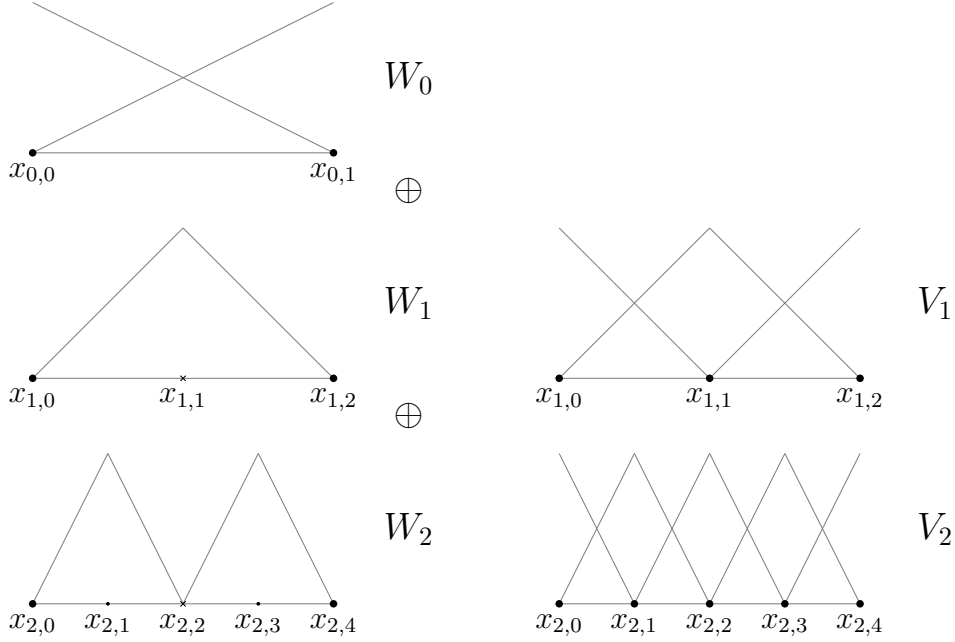
\begin{figure}[hbt!]
	\centering
	\def\PointRadius{0.05}
	\tikzset{
		cross/.pic = {
			\draw[rotate = 45] (-#1,0) -- (#1,0);
			\draw[rotate = 45] (0,-#1) -- (0, #1);
		}
	}
	\begin{tikzpicture}[scale=1]
		\draw[gray] (0, 3) -- ++(4, 0);
		\draw[gray] (0, 3) -- ++(4, 2);
		\draw[gray] (4, 3) -- ++(-4, 2);
		\fill[black]  (0, 3) circle (\PointRadius);
		\fill[black]  (4, 3) circle (\PointRadius);
		\node[below] at (0, 3) {\large$ x_{0, 0}$};
		\node[below] at (4, 3) {\large$ x_{0, 1}$};
		
		\draw[gray] (0, 0) -- ++(4, 0);
		\draw[gray] (0, 0) -- ++(2, 2);
		\draw[gray] (4, 0) -- ++(-2, 2);
		\fill[black]  (0, 0) circle (\PointRadius);
		\fill[black]  (4, 0) circle (\PointRadius);
		\path (2,.0) pic[black] {cross=\PointRadius};
		\node[below] at (0, 0) {\large$x_{1, 0}$};
		\node[below] at (2, 0) {\large$x_{1, 1}$};
		\node[below] at (4, 0) {\large$x_{1, 2}$};
		
		\draw[gray] (0, -3) -- ++(4, 0);
		\draw[gray] (0, -3) -- ++(1, 2);
		\draw[gray] (2, -3) -- ++(-1, 2);
		\draw[gray] (2, -3) -- ++(1, 2);
		\draw[gray] (4, -3) -- ++(-1, 2);
		\fill[black]  (0, -3) circle (\PointRadius);
		\fill[black]  (4, -3) circle (\PointRadius);
		\path (2,.-3) pic[black] {cross=\PointRadius};
		\fill[black] (1, -3) circle ({0.5 * \PointRadius});
		\fill[black] (3, -3) circle ({0.5 * \PointRadius});
		\node[below] at (0, -3) {\large$x_{2, 0}$};
		\node[below] at (2, -3) {\large$x_{2, 2}$};
		\node[below] at (4, -3) {\large$x_{2, 4}$};
		\node[below] at (1, -3) {\large$x_{2, 1}$};
		\node[below] at (3, -3) {\large$x_{2, 3}$};
		
		\node at (5, 3+1) {\Large$W_0$};
		\node at (5,  0+1) {\Large$W_1$};
		\node at (5, -3+1) {\Large$W_2$};
		\node at (5, 1.5+1) {\Large$\oplus$};
		\node at (5, -1.5+1) {\Large$\oplus$};

		\draw[gray] (7, 0) -- ++(4, 0);
		\draw[gray] (7, 0) -- ++(2, 2);
		\draw[gray] (9, 0) -- ++(-2, 2);
		\draw[gray] (9, 2) -- ++(2, -2);
		\draw[gray] (7+2, 0) -- ++(2, 2);		
		\fill[black]  (5+2, 0) circle (\PointRadius);
		\fill[black]  (7+2, 0) circle (\PointRadius);
		\fill[black]  (9+2, 0) circle (\PointRadius);
		\node[below] at (5+2, 0) {\large$x_{1, 0}$};
		\node[below] at (7+2, 0) {\large$x_{1, 1}$};
		\node[below] at (9+2, 0) {\large$x_{1, 2}$};
		
		\draw[gray] (0+7, -3) -- ++(4, 0);
		\draw[gray] (0+7, -3) -- ++(1, 2);
		\draw[gray] (8, -3) -- ++(-1, 2);
		\draw[gray] (10, -3) -- ++(-1, 2);
		\draw[gray] (8, -3) -- ++(1, 2);
		\draw[gray] (10, -3) -- ++(1, 2);
		\draw[gray] (2+7, -3) -- ++(-1, 2);
		\draw[gray] (2+7, -3) -- ++(1, 2);
		\draw[gray] (4+7, -3) -- ++(-1, 2);
		\fill[black]  (0+5+2, -3) circle (\PointRadius);
		\fill[black]  (2+5+2, -3) circle (\PointRadius);
		\fill[black]  (6+2, -3) circle (\PointRadius);	
		\fill[black] (10, -3) circle ({1 * \PointRadius});
		\fill[black] (11, -3) circle ({1 * \PointRadius});
		\node[below] at (7, -3) {\large$x_{2, 0}$};
		\node[below] at (1+5+2, -3) {\large$x_{2, 1}$};
		\node[below] at (2+5+2, -3) {\large$x_{2, 2}$};
		\node[below] at (10, -3) {\large$x_{2, 3}$};
		\node[below] at (11, -3) {\large$x_{2, 4}$};

		\node at (12,  0+1) {\Large$V_1$};
		\node at (12, -3+1) {\Large$V_2$};
		
	\end{tikzpicture}
\caption{An illustration of hierarchical bases for $\ell=0,1,2$.}
\end{figure}

\subsection{Coarse space}
Now we present the construction of the coarse space $V_{h,0}$. The key is to
find local multiscale bases in each coarse subdomain $\Omega_i$, having certain approximation properties to $V_i^{\partial}$. The global multiscale bases functions, obtained from the local ones by the weighted prolongation $\tilde{\mathcal{R}}_i^T$, inherit these approximation properties.

First we define the linear space over the boundary of each subdomain $\partial\Omega_i$. Fix the level parameter $\ell\in \mathbb{N}$, and let $\Gamma_{i}^j$ with $j=1,2,3,4$ be a partition of $\partial\Omega_i$ with no mutual intersection, i.e., $\cup_{j=1}^{4}\overline{\Gamma_{i}^j}=\partial\Omega_i$ and $\Gamma_{i}^j\cap \Gamma_{i}^{j'}=\emptyset$ if $j\neq j'$. Furthermore,
we denote by $V_{i,\ell}^j\subset C(\partial\Omega_{i})$ the linear space spanned by hierarchical bases up to level $\ell$ on each coarse edge $\Gamma_{i}^{j}$ and continuous over $\partial\Omega_i$. The local edge space $V_{i,\ell}$ defined over $\partial\Omega_i$ is the smallest linear space having $V_{i,\ell}^j$ as a subspace. Let $\{\psi_{i,\ell}^j\}_{j=1}^{2^{\ell+2}}$ and 
$\{x_{i,\ell}^j\}_{j=1}^{2^{\ell+2}}$ be the nodal bases and the associated nodal points for $V_{i,\ell}$. Then we can represent the local edge space $V_{i,\ell}$ by  
\begin{align}\label{eq:local-edge}
V_{i,\ell} := \text{span} \big\{\psi_{i,\ell}^j : 1 \leq j \leq  2^{\ell+2}\big\}.
\end{align}
We depict in Figure \ref{fig:grid-points-local} the grid points $\{x_{i,\ell}^j\}_{j=1}^{2^{\ell+2}}$ of the hierarchical bases over $\partial\Omega_i$ for level parameter $\ell\in\{0,1,2\}$. 

\begin{figure}[hbt!]
	\centering
	\begin{tikzpicture}[scale=1]
		\draw[step=1.0, black, very thick] (-0, -0) grid (4, 4);
		\foreach \x in {0,...,4}
		\foreach \y in {0,...,4}{
			\fill (1.0 * \x, 1.0 * \y) circle (1.5pt);
		}
		\fill [red] (2.0 , 2.0 ) circle (1.5pt);
		\fill[green, opacity=0.4] (1.0, 1.0) rectangle (3.0, 3.0);
		\node at (1.8, 1.8) {$\Omega_i$};
		
		\draw[step=1.0, black, very thick] (5,4) grid (7, 6);
		\foreach \x in {5,7}
		\foreach \y in {4,6}{
			\fill [blue] (1.0 * \x, 1.0 * \y) circle (1.5pt);
		}
		
		\draw[step=1.0, black, very thick] (5,1) grid (7, 3);
		
		\foreach \x in {5,6,7}
		\foreach \y in {1,2,3}{
			\fill [blue] (1.0 * \x, 1.0 * \y) circle (1.5pt);
		}

		\draw[step=1.0, black, very thick] (5,-2) grid (7, 0);
		
		\foreach \x in {5,6,7}
		\foreach \y in {-2,-1,0}{
			\fill [blue] (1.0 * \x, 1.0 * \y) circle (1.5pt);	
		}
		
		\foreach \x in {5.5,6.5}
		\foreach \y in {-2,0}{
			\fill [blue] (1.0 * \x, 1.0 * \y) circle (1.5pt);		
		}
		
		\foreach \x in {5,7}
		\foreach \y in {-1.5,-0.5}{
			\fill [blue] (1.0 * \x, 1.0 * \y) circle (1.5pt);
		}
		
		\fill [red] (6,-1) circle (1.5pt);
		\fill [red] (6,2) circle (1.5pt);
		\fill [red] (6,5) circle (1.5pt);
		
		\draw [-to](5,4.6) -- (3.0,2.6);
		\draw [-to](5,2.1) -- (3.0,2.1);
		\draw [-to](5,-0.4) -- (3.0,1.4);
		
		\node at (7.8, -1) {$\ell=2$};
		\node at (7.8, 2) {$\ell=1$};
		\node at (7.8, 5) {$\ell=0$};
	\end{tikzpicture}
	\caption{Grid points for $\ell=0,1,2$ over $\partial\Omega_i$.}
\label{fig:grid-points-local}
\end{figure}
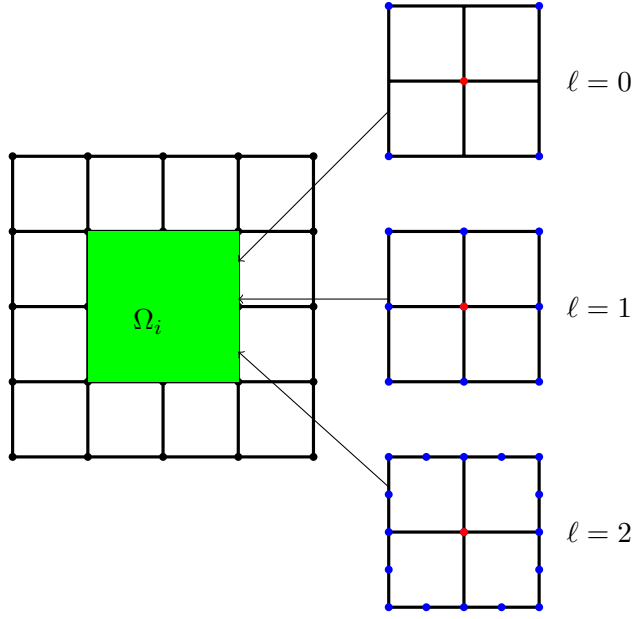

Next, we define the local multiscale space over each coarse neighborhood $\Omega_i$, defined by 
\begin{align}\label{eq:local-multiscale}
\mathcal{L}^{-1}_i(V_{i,\ell})
:= \text{span} \left\{\mathcal{L}^{-1}_i(\psi_{i,\ell}^j),v^i : 1 \leq j \leq  2^{\ell+2}\right\},
\end{align}
where $\mathcal{L}^{-1}_i (\psi_{i,\ell}^j):=v\in H^1(\Omega_i)$ is the solution to the following local problem
\begin{equation}
\label{eq:Li}
  \left\{ \begin{aligned}
          \mathcal{L}_i v&:=\Delta v+k^2 v=0, &&\mbox{in }\Omega_i,\\
          v&=\psi_{i,\ell}^j, &&\mbox{on }\partial\Omega_i.
  \end{aligned}\right.
\end{equation}
Then we compute one local solution $v^i$ defined by the solution to the local problem
\begin{equation}\label{eq:Li-local}
  \left\{ \begin{aligned}
          \mathcal{L}_i v^i&:=\Delta v^i+k^2 v^i=1, &&\mbox{in }\Omega_i,\\
          v^i&=0, &&\mbox{on }\partial\Omega_i.
  \end{aligned}\right.
\end{equation}
By the construction, the dimension of $\mathcal{L}^{-1}_i (V_{i,\ell})$ is $2^{\ell+2}+1$. In practice, $2^{\ell+2}+1$ local problems are solved by the FEM to obtain the local multiscale space $\mathcal{L}^{-1}_i(V_{i,\ell})$, which can be carried out in parallel and thus has a low computational complexity. Note that the local problems \eqref{eq:Li} and \eqref{eq:Li-local} may admit more than one solution when $k$ is the eigenvalue of the corresponding Dirichlet Laplacian. Nonethelss, one can make it unique by taking the local solution with the minimum $V(\Omega_i)$-norm. 

Finally, the edge multiscale ansatz space $V_{h,0}$ is defined by the weighted prolongation $\tilde{\mathcal{R}}_i^T$
\begin{align}\label{eq:global-multiscale}
V_{h,0} := \text{span} \left\{\tilde{\mathcal{R}}_i^T\mathcal{L}^{-1}_i(\psi_{i,\ell}^j), \tilde{\mathcal{R}}_i^Tv^i: \,  \, 1 \leq i \leq N \text{ and } 1 \leq j \leq  2^{\ell+2}\right\}.
\end{align}
The construction of $V_{h,0}$ is summarized in Algorithm \ref{algorithm:wavelet}. 

\begin{algorithm}[H]
\caption{The construction of the edge multiscale space $V_{h,0}$.}
\label{algorithm:wavelet}
    \KwData{The level parameter $\ell\in \mathbb{N}$; coarse neighborhood $\Omega_i$ and its four coarse edges $\Gamma_{i}^{j}$ with
    $j=1,2,3,4$; the subspace $V_{\ell,i}^k\subset L^2(\Gamma_{i}^{k})$ up to level $\ell$ on each coarse edge $\Gamma_{i}^{k}$.
    }
    \KwResult{$V_{h,0}$}
Construct the local edge space $V_{i,\ell}$ \eqref{eq:local-edge}\;    
Compute the local multiscale space $\mathcal{L}^{-1}_i (V_{i,\ell})$ \eqref{eq:local-multiscale}\;
Compute one specific local multiscale basis function $v_i$ from \eqref{eq:Li-local}\;
Construct the coarse space $V_{h,0}$ \eqref{eq:global-multiscale}.
\end{algorithm}

Next we give the coarse projection operator $\mathcal{P}_0$ of level $\ell$: $V_h\to V_{h,0}$. Since $V_{h,0}$ is generated by the local multiscale space $\mathcal{L}^{-1}_i(V_{i,\ell})$ by means of the weighted prolongation $\tilde{\mathcal{R}}_i^T$, we only need to define local projection operator $\mathcal{P}_{i,\ell}$ at level $\ell$: $L^2(\partial\Omega_i)\to \mathcal{L}^{-1}(V_{i,\ell})$. Then $\mathcal{P}_{i,\ell}: L^2(\partial\Omega_i)\to \mathcal{L}_i^{-1}(V_{i,\ell})$ can be defined by 
\begin{align}\label{eq:projectionEDGE}
\mathcal{P}_{i,\ell}v:=\mathcal{L}_{i}^{-1}(\mathcal{I}_{i,\ell}(v|_{\partial\Omega_i})).
\end{align}
Here, $\mathcal{I}_{i,\ell}:L^2(\partial\Omega_i)\to V_{i,\ell}$ denotes the $L^2$-projection. 

Note that $\mathcal{P}_{i,\ell}v|_{\partial\Omega_i}$ is the $L^2$-projection on $V_{i,\ell}$:
\begin{align*}
\mathcal{P}_{i,\ell}v|_{\partial\Omega_i}
=\mathcal{I}_{i,\ell}(v|_{\partial\Omega_i}).
\end{align*}
Since any $v\in V_h$ can be represented by
$v=\sum_{i=1}^N \tilde{{R}}_i^T v|_{\Omega_i}$,
the coarse projection $\mathcal{P}_0$ of level $\ell$: $ V_h\to V_0$ can be defined in terms of the local projection \eqref{eq:projectionEDGE},
\begin{align}\label{eq:glo-proj}
\mathcal{P}_0(v):=\sum_{i=1}^N\tilde{{R}}_i^T \left(\mathcal{P}_{i,\ell}v\right).
\end{align}

\section{Convergence analysis}\label{sec:converge}
Now we prove that the operator $I-B_{\op{hybrid}}A$ is contractive under very mild conditions on the coarse mesh size $H$ and level parameter $\ell$ in \eqref{eq:wavelet-level-H} and \eqref{eq:wavelet-level-ell}. This is achieved as follows. We first establish the approximation property of the edge multiscale ansatz space $V_{h,0}$ to the target Helmholtz harmonic linear space $V^{\partial}$ and the solution to the adjoint problem \eqref{eqn:weakformDual} in Lemma \ref{thm:last}. Then we establish that $\|I-R_0^TQ_0\|$ is bounded  by 1/2 using the classical Schatz argument which combines with G\aa rding's inequality \eqref{eq:garding} and the approximation property of $V_{h,0}$ to the solution of the adjoint problem \eqref{eqn:weakformDual}. 

\begin{lemma}[Approximation properties of the projection $\mathcal{P}_{i,\ell}$] \label{thm:proj}
Let Assumption \ref{ass:resolution} hold. For all $v_i^{\partial}\in V^{\partial}_i$, let $e_i\in V_{h,i}$ satisfy
\begin{equation}\label{eq:pde-approx}
\left\{
\begin{aligned}
\mathcal{L}_i e_i&:=\Delta e_i+k^2 e_i=0,&&\mbox{in }\Omega_i,\\
 e_i&=v_{i}^{\partial}-\mathcal{P}_{i,\ell}(v_{i}^{\partial}),&&\mbox{on }\partial\Omega_i.
\end{aligned}
\right.
\end{equation}
Then there holds 
\begin{align*}
\normL{e_i}{\Omega_i}&\leq\Cw \Ce2^{-\ell/2}d_i\| v_{i}^{\partial}\|_{H^1( \Omega_i)},\\
\normL{\nabla(\chi_i\locv{e}{}{i})}{\Omega_i}
&\leq\Cw\Ce 2^{-\ell/2}\delta^{-1} d_i\| v_{i}^{\partial}\|_{H^1( \Omega_i)}.
\end{align*}
\end{lemma}
\begin{proof}
The proof is a slight modification of that for \cite[equation (4.5)]{fu2021wavelet} by changing the minimum overlapping size from $H$ to $\delta$, and using Propostion \ref{prop:approx-wavelets}. 
\end{proof}

Next, we derive approximation properties of the edge multiscale ansatz space $V_{h,0}$ to the target Helmholtz harmonic linear space $V^{\partial}$ and the solution to the adjoint problem \eqref{eqn:weakformDual}. Similar to \cite[Theorem 4.2]{fu2021wavelet}, the proof is based on the global and local splittings, but utilizing a weighted prolongation $\tilde{R}^T_i$ instead of the partition of the unity functions $\chi_i$. 
\begin{lemma}[Approximation properties of the edge multiscale ansatz space $V_{h,0}$] \label{thm:last}
Let \eqref{eq:c-diameter} and Assumption \ref{ass:resolution} hold. 
For any $v^{\partial}\in V^{\partial}$ defined in \eqref{eq:glo-edge-space}, there exists $\Const{\rm ap}$  independent of $(h,k,H)$ such that
\begin{equation}\label{eq:waveletErrconv1}
\|v^{\partial}-\mathcal{P}_0 v^{\partial}\|_V
\leq \Const{\rm ap} \Ce\delta^{-1} H2^{-\ell/2}\|{v^{\partial}}\|_{V}.
\end{equation} 
For all $w_h\in V_h$, let $z_h\in V_h$ be the solution to the adjoint problem \eqref{eqn:weakformDual}. Then there holds
\begin{equation}\label{eq:waveletErrconvDual}
\inf\limits_{v_0\in V_{h,0}}\|z_h-v_0\|_V\leq
\Const{\ref{thm:last}}\delta^{-1}{H}
(H+\Ce^2 2^{-\ell/2})\|{w_h}\|_{L^2(\Omega)}.
\end{equation}

\end{lemma}
	\begin{proof}
 By \eqref{eq:glo-edge-space} and  \eqref{eq:glo-proj}, each $v^{\partial}\in V^{\partial}$ and its coarse projection $\mathcal{P}_0(v^{\partial})$ at level $\ell$ can be expressed 
\begin{align}\label{eq0}
v^{\partial}&:=\sum_{i=1}^N \tilde{\RR}^T_{i}v_{i}^{\partial}\quad \mbox{and}\quad \mathcal{P}_0(v^{\partial})=\sum_{i=1}^N\tilde{{R}}_i^T \left(\mathcal{P}_{i,\ell}v^{\partial}|_{\partial\Omega_i}\right).
\end{align}
Let $e:=v^{\partial}-\mathcal{P}_0 v^{\partial}$. Then 
$e=\sum_{i=1}^N \tilde{\RR}^T_{i}e_{i}$, with $ e_{i}:=v^{\partial}_i-\mathcal{P}_{i,\ell} v^{\partial}|_{\partial\Omega_i}$.
Since $e_i$ satisfies \eqref{eq:pde-approx},  Lemma \ref{thm:proj}, \eqref{eq0} and Theorem \ref{thm:pum-2} lead to 
\begin{align*}
\|e\|_{L^2(\Omega)}&\lesssim \left(\sum_{i=1}^N\|e_i\|_{L^2(\Omega_i)}^2+h^2\|\nabla (\chi_ie_i)\|_{L^2(\Omega_i)}^2\right)^{1/2}
\lesssim\Ce2^{-\ell/2}d_i \left(\sum_{i=1}^N\| v_{i}^{\partial}\|_{H^1( \Omega_i)}^2\right)^{1/2},\\
\|\nabla e\|_{L^2(\Omega)}&\lesssim  \left(\sum_{i=1}^N\|\nabla (\chi_ie_i)\|_{L^2(\Omega_i)}^2\right)^{1/2}
\lesssim\Ce 2^{-\ell/2}\delta^{-1}{d_i}\left(\sum_{i=1}^N 
\| v_{i}^{\partial}\|^2_{H^1( \Omega_i)}\right)^{1/2}.
\end{align*}
By combining these two estimates with Assumption \ref{ass:resolution}, we prove \eqref{eq:waveletErrconv1}. 
Next, we prove \eqref{eq:waveletErrconvDual}. Note that
\begin{align*}
z_h&=\sum_{i=1}^N\widetilde{\mathcal{R}}_i^{T}(z_h|_{\Omega_i})
=\sum_{i=1}^N\widetilde{\mathcal{R}}_i^{T}\left(z_h^{i,\roma}+z_h^{i,\romb}+v^{i}\dashint_{\Omega_i}w_h{\rm d}x\right)
\end{align*}
with $z_h^{i,\roma}:=z_h|_{\Omega_i}-(z_h^{i,\romb}+v^{i}\dashint_{\Omega_i}w_h{\rm d}x)$ and $\dashint_{\Omega_i}v{\rm d}x:=|\Omega_i|^{-1}\int_{\Omega_i}v{\rm{d}}x$ denotes the average of the function $v\in L^1(\Omega_i)$ over each subdomain $\Omega_i$. 
$z_h^{i,\romb}\in V_{h,i}$ is the solution to the following problem 
	\begin{equation*}
		\left\{
		\begin{aligned}
			a_i(z_h^{i,\romb},v_i)&=0,&&\forall v_i\in V_{h,i}^0,\\
			z_h^{i,\romb}&=z_{h}|_{\partial\Omega_i},&&\mbox{on }\partial\Omega_i.
		\end{aligned}
		\right.
	\end{equation*}
Note that $z_h^{i,\roma}\in V_{h,i}^0$ satisfies 
\begin{equation*}	
a_i(z_h^{i,\roma},v_i)=\left(v_i,w_h-\dashint_{\Omega_i}w_h{\rm d}x\right),\mbox{ }\forall v_i\in V_{h,i}^0.
\end{equation*}
Let $v_0:=\sum_{i=1}^N\tilde{{R}}_i^T (\mathcal{P}_{i,\ell}z_h|_{\partial\Omega_i}+v^{i}\dashint_{\omega_i}w_h)\in V_{h,0}$ be an approximation to $z_h$. 
Direct computation gives 
	\begin{align}
		\Vert z_h-v_0\Vert^2_V&=\left\Vert \sum_{i=1}^N\tilde{{R}}_i^T z_h^{i,\roma}+\sum_{i=1}^N\tilde{{R}}_i^T(z_h^{i,\romb}-\mathcal{P}_{i,\ell}z_h^{i,\romb}|_{\partial\Omega_i})\right\Vert_V^2\nonumber\\
&\leq 2\left\Vert\sum_{i=1}^N\tilde{{R}}_i^T z_h^{i,\roma}\right\Vert_{V}^2
+2\left\Vert\sum_{i=1}^N\tilde{{R}}_i^T(z_h^{i,\romb}-\mathcal{P}_{i,\ell}z_h^{i,\romb}|_{\partial\Omega_i})\right\Vert_{V}^2.\label{eq2_4}
\end{align}
Next we bound the two terms. By \cite[Lemma 4.1]{fu2021wavelet}, we have 
 \begin{align}\label{eq:loc-bubble}
d_i\Vert \nabla z_h^{i,\roma}\Vert_{L^2(\Omega_i)}+\Cpoin{\Omega_i}^{-1/2}\Vert z_h^{i,\roma}\Vert_{L^2(\Omega_i)}
\leq \Cpoin{\Omega_i}^{1/2}d_i^2\Vert w_h\Vert_{L^2(\Omega_i)}.
 \end{align}
This, together with Theorem \ref{thm:pum-2} and the finite-overlap assumption, yields
\begin{align}\label{eq:local-estimate} 
\left\Vert\sum_{i=1}^N\tilde{{R}}_i^T z_h^{i,\roma}\right\Vert_{V}\lesssim 
\delta^{-1} d_i^2 \Vert w_h\Vert_{L^2(\Omega)}.
\end{align}
Meanwhile, Lemma \ref{thm:proj} and the definition of $z_h^{i,\romb}$ indicate 
\begin{equation}\label{eq:lastsss}
\begin{aligned}
\normL{z_h^{i,\romb}-\mathcal{P}_{i,\ell}z_h^{i,\romb}|_{\partial\Omega_i}}{\Omega_i}
&\leq \Cw \Ce2^{-\ell/2}d_i\| z_h^{i,\romb}\|_{H^1( \Omega_i)},\\
\normL{\nabla(\chi_i(z_h^{i,\romb}-\mathcal{P}_{i,\ell}z_h^{i,\romb}|_{\partial\Omega_i}))}{\Omega_i}
&\leq \Cw \Ce2^{-\ell/2}\delta^{-1}d_i\| z_h^{i,\romb}\|_{H^1( \Omega_i)}.
\end{aligned}
\end{equation}
Moreover, by \cite[Proof to Theorem 4.3]{fu2021wavelet}, we obtain 
\begin{align}\label{eq:local-estimate2}
\left\Vert z_h^{i,\romb}\right\Vert_{H^1(\Omega_i)}
\leq \left\Vert z_h\right\Vert_{H^1(\Omega_i)}+\left\Vert z_h-z_h^{i,\romb}\right\Vert_{H^1(\Omega_i)}
\lesssim 
\left\Vert  z_h\right\Vert_{H^1(\Omega_i)}+
\Ce \pi^{-1}d_i\Vert w_h\Vert_{L^2(\Omega_i)}.
\end{align}
To derive the {\em a priori} estimate for $z_h$, Theorem \ref{thm_1} and \eqref{eqn:weakformDual} imply
	\begin{align*}
	&	\Vert z_h\Vert_V
  \lesssim k\Const{stab}\sup_{v_h\in V_h}\frac{|{a}^*(z_h,v_h)|}{\Vert v_h\Vert_{V}}
         =k\Const{stab}\sup_{v_h\in V_h}\frac{|a(v_h,z_h)|}{\Vert v_h\Vert_{V}}\nonumber\\
         =&k\Const{stab}\sup_{v_h\in V_h}\frac{|(w_h,z_h)|}{\Vert v_h\Vert_{V}}
	\leq k\Const{stab}\sup_{v_h\in V_h}\frac{\Vert v_h\Vert_{L^2(\Omega)}\Vert w_h\Vert_{L^2(\Omega)}}{\Vert v_h\Vert_V}
		\leq\Const{stab}\Vert w_h\Vert_{L^2(\Omega)}.\nonumber
\end{align*}
This, \eqref{eq:lastsss}, Theorem \ref{thm:pum-2}, \eqref{eq:local-estimate2}, the finite-overlap assumption, Assumption \ref{ass:resolution} and \eqref{eq:c-diameter} yield
\begin{align*}
&\left\Vert\sum_{i=1}^N\tilde{{R}}_i^T
(z_h^{i,\romb}-\mathcal{P}_{i,\ell}z_h^{i,\romb})\right
\Vert_{V}^2
\lesssim\Ce^2 2^{-\ell}\left(k^2H^2+ (\delta^{-1}H)^2\right)\sum_{i=1}^N\| z_h^{i,\romb}\|_{H^1(\Omega_i)}^2\nonumber\\
\lesssim&\Ce^2 2^{-\ell}\left(k^2H^2+(\delta^{-1}H)^2\right)\left(\| z_h\|_{H^1(\Omega)}^2+\Ce^2 d_i^2\Vert w_h\Vert_{L^2(\Omega)}^2\right)
\lesssim\Ce^4 2^{-\ell}(\delta^{-1}H)^2\Vert w_{h}\Vert_{L^2(\Omega)}^2,
\end{align*}
This, \eqref{eq2_4}--\eqref{eq:local-estimate} and \eqref{eq:c-diameter} show the desired result. 
\end{proof}

\begin{proposition}[Estimate for $I-R_0^TQ_0$]\label{prop:wavelet-basedconv}
Let the fine grid mesh size $h$, the coarse grid mesh size $H$ and the diameter $d_i$ of each subdomain $\Omega_i$ satisfy \eqref{eq:fine-mesh-cond}, Assumption \ref{ass:resolution} and \eqref{eq:c-diameter}, respectively. Moreover, assume that the coarse mesh size $H$ and the level parameter $\ell\in \mathbb{N}_{+}$ satisfy
\begin{align}
H&\leq \delta^{1/2}(2k\Const{b}\Const{\ref{thm:last}}\Ce^2)^{-1/2},\label{eq:wavelet-level-H}\\
\ell&\geq 2\log_2\big(\Const{b}C_{\partial}\Const{\rm ap}\Ce \delta^{-1}H(kH)^{-3/2}(kh)^{-1/2}(k\delta)^{-1}\big)+4.\label{eq:wavelet-level-ell}
\end{align}
Let $v_h\in V_h$, and let $v_h^{\partial}$ be defined in \eqref{eq:proj-edge-err}. There holds
\begin{equation}\label{eq:waveletErrconv}
\|(I-R_0^TQ_0)v_h^{\partial}\|_V
\leq \tfrac{1}{2}\|{v_h}\|_{V}.
\end{equation}
\end{proposition}
\begin{proof}
Let $e_h:=(I-R_0^TQ_0)v_h^{\partial}\in V_0\subset V$.
G\aa rding\rq{}s inequality in Theorem \ref{them:sesquilinear} implies
\begin{align*}
\|e_h\|_V^2&\leq \mathrm{Re} \{a(e_h,e_h)\}
+2k^2\normL{e_h}{\Omega}^2.
\end{align*}
The representation $e_h=({A}^{-1}-R_0^T{A}_0^{-1}R_0){A}v_h^{\partial}$, the inclusion $V_{h,0}\subset V_h$ and Galerkin orthogonality imply
\begin{align}\label{eq:222}
\|e_h\|_V^2
&=\mathrm{Re} \{a(e_h,v_h^{\partial}-v_{0})\}
+2k^2\normL{e_h}{\Omega}^2, \quad\forall v_0\in V_{h,0}.
\end{align}
Next, we estimate $\normL{e_h}{\Omega}$ by the Aubin-Nitsche technique. Let $z_h\in V_h$ be the solution to
\begin{align*}
a^*(z_h,v_h)=(v_h,e_h)_{\Omega}, \quad\forall v_h\in V_h. 
\end{align*}
For all $z_0\in V_0$, we obtain
\begin{align*}
\normL{e_h}{\Omega}^2
&=a^*(z_h,e_h)=\overline{a(e_h,z_h)}
=\overline{a(e_h,z_h-z_0)}.
\end{align*}
Furthermore, the estimates \eqref{eq:sesqui-bound} and \eqref{eq:waveletErrconvDual} lead to
\begin{align*}
\normL{e_h}{\Omega}^2
&\leq \Const{b}\|e_h\|_V\inf\limits_{z_0\in V_0}\|z_h-z_0\|_V\leq  \Const{b}\Const{\ref{thm:last}}\delta^{-1}{H}\big(H+\Ce^2 2^{-\ell/2}\big)\normL{e_h}{\Omega}\|e_h\|_V.
\end{align*}
This implies 
\begin{align*}
\normL{e_h}{\Omega}
\leq  \Const{b}\Const{\ref{thm:last}}\delta^{-1}{H}\big(H+\Ce^2 2^{-\ell/2}\big)\|e_h\|_V.
\end{align*}
By condition \eqref{eq:wavelet-level-H}, we arrive at
$\normL{e_h}{\Omega}\leq \tfrac{1}{2k}\|e_h\|_V$.
This and \eqref{eq:222} lead to
\begin{align*}
\|e_h\|_V^2\leq \mathrm{Re}\{a(e_h,v_h^{\partial}-v_0)\}+\tfrac{1}{2}\|e_h\|_V^2.
\end{align*}
Consequently, we obtain
$\|e_h\|_V^2\leq 2 \mathrm{Re} \{a(e_h,v_h^{\partial}-v_0)\}$.
Finally,  the boundedness of the sesquilinear form $a(\cdot,\cdot)$ in Theorem \ref{them:sesquilinear} implies
\begin{align*}
\|e_h\|_V\leq 2\Const{b} \inf_{v_0\in V_0}\|v_h^{\partial}-v_0\|_V.
\end{align*}
Then the approximation property \eqref{eq:waveletErrconv1} and the estimate \eqref{eq:vhpar} indicate
\begin{align*}
\|e_h\|_V\leq 2\Const{b}C_{\partial} \Const{\rm ap} \Ce 2^{-\ell/2}\delta^{-1} H\sigma\|{u_h}\|_{V}.
\end{align*}
Hence, the desired assertion follows from condition \eqref{eq:wavelet-level-ell}.
\end{proof}
\begin{remark}[Conditions \eqref{eq:wavelet-level-H} and \eqref{eq:wavelet-level-ell}]\label{rm:l}
Given a wavenumber $k$, let $h:=k^{-3/2}$ and the minimum overlapping size $\delta:=h$, then Conditions \eqref{eq:wavelet-level-H} and \eqref{eq:wavelet-level-ell} require
$H\approx k^{-5/4}\text{ and }\ell\approx \log_2(k^{11/8})$.
If one increases the minimum overlapping size to be $\delta\approx H$, then Conditions \eqref{eq:wavelet-level-H} and \eqref{eq:wavelet-level-ell} require
$H\approx k^{-1}\text{ and }\ell\approx \log_2(k^{1/4})$.
Conditions \eqref{eq:wavelet-level-H} and \eqref{eq:wavelet-level-ell} are ensured in both cases by a mild resolution condition on the coarse grid size $H$ and the level parameter $\ell$. Decreasing the minimum overlapping size $\delta$ implies smaller local problems, which, however, requires a finer $H$ and a larger $\ell$.
\end{remark}
Finally, we derive the contractivity of the operator $I-B_{\op{hybrid}}A $.
\begin{theorem}
Under the assumptions of Proposition \ref{prop:wavelet-basedconv}, there holds 
$\|I-B_{\op{hybrid}}A\|\leq \frac{1}{2}$.
\end{theorem}

\section{Numerical tests}\label{sec:num}
In this section we verify key estimates in Section \ref{sec:converge} and report the performance of the proposed EMs-HS preconditioner $B_{\op{hybrid}}$. We employ two settings of $(h,H)$ for a fixed wavenumber $k$. {\em Setting A} corresponds to 
$h \simeq k^{-3/2}$ and $ H=1/k,$ 
where the notation $\simeq$ means that $h$ is chosen as close to $k^{-3/2}$ as possible.
{\em Setting B} takes a fixed number $m$ of grid-points per wavelength,
$h =2\pi/(mk)$ and $H=2/k$,
with $m\in\mathbb{N}_+$. Throughout, we take $m=6\pi$ and hence $h =\frac{1}{3k}$ in {\em Setting B}. The degrees of freedom of the linear system $A$ is $(1/h+1)^d$ and the dimension of coarse space 
$V_{h,0}$ is the product of the number of local multiscale basis functions with the number of coarse element, i.e. $(2^{\ell+d}+1)(1/H)^d$.
	
\begin{table}[hbt!]
		\centering \begin{tabular}{ccccccccc}
			\toprule 
			\multicolumn{5}{c}{$h\simeq k^{-3/2}$, $H=1/k$}\tabularnewline
			\midrule 
			$k$& $h$ & $H$ & $ {(hk)}^{-\frac{1}{2}}$& $\widehat{\sigma}$\tabularnewline\hline
			10&$1/40$&1/10&  2.00&  1.40  \tabularnewline\hline 
			20&$1/100$&1/20& 2.24  &  1.52  \tabularnewline\hline
			40&$1/280$&1/40&2.65   &  1.72  \tabularnewline\hline
			80&$1/720$&1/80& 3.00  &   1.90 \tabularnewline\hline
			160&1/2080 &1/160&  3.61 &  2.19  \tabularnewline\hline
			320&1/5760&1/320&  4.24 &  2.49  \tabularnewline  
	\bottomrule
        \end{tabular} \quad  
        \begin{tabular}{ccccccccc}
        \toprule
			\multicolumn{5}{c}{$h= 1/(3k)$, $H=2/k$}\tabularnewline\midrule
			$k$& $h$ & $H$ &  $ {(hk)}^{-\frac{1}{2}}$& $\widehat{\sigma}$\tabularnewline\hline
			10&1/30& 1/5& 1.73 &1.25   \tabularnewline\hline
			20&1/60& 1/10& 1.73 & 1.25   \tabularnewline\hline
			40&1/120& 1/20&1.73 & 1.25  \tabularnewline\hline
			80&1/240& 1/40& 1.73 & 1.25   \tabularnewline\hline
			160&1/480& 1/80& 1.73 & 1.25  \tabularnewline\hline
			320&1/960&1/160 & 1.73 & 1.25 \tabularnewline	\bottomrule		
		\end{tabular}
		\caption{Comparison of computed value $\widehat{\sigma}$ with predicted value $ {(hk)}^{-\frac{1}{2}}$ for various $h$.}
		\label{ta:test1}
	\end{table}
	
\subsection{Test 1: verification of Lemma \ref{lem:decomp}}
This test is to verify Lemma \ref{lem:decomp}, or \eqref{eq:local-harmonic}. Thus we compare the quantity 
$\widehat{\sigma}:=\max_{i=1,\cdots,N}\frac{\left\|Q_i v_h-v_h|_{\Omega_i}\right\|_{V(\Omega_i)}}{\|v_h\|_{V(\widetilde{\Omega}_{i})}}$ with the predicted value $ {(hk)}^{-\frac{1}{2}}$ from \eqref{eq:local-harmonic}, which are presented in Table \ref{ta:test1}. 
The correlation coefficient between $\widehat{\sigma}$ and ${(hk)}^{-\frac{1}{2}}$ is 0.9999 in {\em Setting A}. Note that the quantity $ {(hk)}^{-\frac{1}{2}}$ takes one unique value in {\em Setting B}, and the computed value $\widehat{\sigma}$ is always independent of $k$. Altogether, these results indicate that the estimate in Lemma \eqref{lem:decomp} is sharp.	

\subsection{Test 2: verification of Lemma \ref{thm:proj}}
Now we verify the estimates in Lemma \ref{thm:proj}. Table \ref{ta:test2a} shows 
$\Delta_1=:\max_{i=1,\cdots,N}\frac{ \| e_i\|_{L^2(\Omega_i)} }{ \|  v_i^{\partial}\|_{H^1(\Omega_i)} }$ and $2^{-\ell/2}H$ proved in Lemma \ref{thm:proj}, and Table \ref{ta:test2b} presents 
$\Delta_2:=\max_{i=1,\cdots,N}\frac{ \|\chi_i \nabla e_i\|_{L^2(\Omega_i)} }{ \|v_i^{\partial}\|_{H^1(\Omega_i)} }$ 
against the proved value $2^{-\ell/2}H/h$. Table \ref{ta:test2a} indicates a first-order convergence of $\Delta_1$ with respect to $H$ in {\em Setting A}. Although the dependence of $\Delta_2$ on $H/h$ in Table \ref{ta:test2b} is not significant in {\em Setting A}, the same observation holds. Moreover, in {\em Setting B}, Table \ref{ta:test2a} and Table \ref{ta:test2b} indicate $\Delta_1=\mathcal{O}(H)$ and $\Delta_2=\mathcal{O}(H/h)$ for a fixed level parameter $\ell$. 

\begin{table}[hbt!]
\centering
\begin{tabular}{ccccccccc}
\toprule 
\multicolumn{9}{c}{$h\simeq k^{-3/2}, H=1/k$}\tabularnewline
\midrule 
				\multirow{2}{*}{$k$}&  \multirow{2}{*}{$h$}&  \multirow{2}{*}{$H$}& \multicolumn{2}{c}{$\ell=0$} & \multicolumn{2}{c}{$\ell=1$} & \multicolumn{2}{c}{$\ell=2$}\tabularnewline
				\cmidrule{4-9} \cmidrule{5-9} \cmidrule{6-9} \cmidrule{7-9} \cmidrule{8-9} \cmidrule{9-9}
				&&	&$\Delta_1$& $2^{-\ell/2}H$ & $\Delta_1$& $2^{-\ell/2}H$ &$\Delta_1$& $2^{-\ell/2}H$\tabularnewline
				\midrule 
				10&1/40& 1/10&  0.014488  &0.1000   & 0.012014 & 0.0707 &0.010737 & 0.0500\tabularnewline	\midrule 
				20&1/100&1/20 &  0.005940 &0.0500   &0.005096 & 0.0354 &0.004565 & 0.0250\tabularnewline\midrule
				40&1/280&1/40 &  0.002194  & 0.0250   & 0.001965 & 0.0177 &0.001720 & 0.0125\tabularnewline\midrule
				80&1/720& 1/80&  0.000871 & 0.0125  & 0.000801 & 0.0088 &0.000704 & 0.0063\tabularnewline\midrule
				160&1/2080&1/160 & 0.000309 & 0.0063   & 0.000293 & 0.0044 &0.000263 & 0.0031\tabularnewline\midrule
				320&1/5760&1/320 & 0.000114 &  0.0031   &  0.000110 & 0.0022& 0.000098 & 0.0016\tabularnewline
				\bottomrule
				
				\multicolumn{9}{c}{$h= 1/(3k)$, $H=2/k$}\tabularnewline
				\midrule 
				\multirow{2}{*}{$k$}&  \multirow{2}{*}{$h$}&  \multirow{2}{*}{$H$}& \multicolumn{2}{c}{$\ell=0$} & \multicolumn{2}{c}{$\ell=1$} & \multicolumn{2}{c}{$\ell=2$}\tabularnewline
				\cmidrule{4-9} \cmidrule{5-9} \cmidrule{6-9} \cmidrule{7-9} \cmidrule{8-9} \cmidrule{9-9}
				&&	&$\Delta_1$& $2^{-\ell/2}H$ & $\Delta_1$& $2^{-\ell/2}H$ &$\Delta_1$& $2^{-\ell/2}H$\tabularnewline
				\midrule 
				10&1/30& 1/5 & 0.0219196 & 0.200000   & 0.0181578 & 0.141421  & 0.0147435 & 0.100000 \tabularnewline	\midrule 
				20&1/60& 1/10&  0.0109598 & 0.100000 & 0.0090789 & 0.070711  & 0.0073717 & 0.050000\tabularnewline\midrule
				40&1/120& 1/20& 0.0054799 & 0.050000   & 0.0045394 & 0.035355  & 0.0036859 & 0.025000\tabularnewline\midrule
				80&1/240& 1/40& 0.0027399 & 0.025000  & 0.0022697 & 0.017678  &0.0018429 & 0.012500\tabularnewline\midrule
				160&1/480&1/80 & 0.0013700 & 0.012500   & 0.0011349 & 0.008839  & 0.0009215 & 0.006250\tabularnewline\midrule
				320&1/960& 1/160& 0.0006850 & 0.006250   & 0.0005674 & 0.004419  & 0.0004607 & 0.003125\tabularnewline
				\bottomrule
				
			\end{tabular}
		\caption{Comparion of computed value $\Delta_1$ with $2^{-\ell/2}H$. }
			\label{ta:test2a}
	\end{table}
	
\begin{table}[hbt!]
\centering
\begin{tabular}{ccccccccc}
				\toprule 
				\multicolumn{9}{c}{$h\simeq k^{-3/2}, H=1/k$}\tabularnewline
				\midrule 
				\multirow{2}{*}{$k$}&  \multirow{2}{*}{$h$}&  \multirow{2}{*}{$H$}& \multicolumn{2}{c}{$\ell=0$} & \multicolumn{2}{c}{$\ell=1$} & \multicolumn{2}{c}{$\ell=2$}\tabularnewline
				\cmidrule{4-9} \cmidrule{5-9} \cmidrule{6-9} \cmidrule{7-9} \cmidrule{8-9} \cmidrule{9-9}
				&&	&$\Delta_2$&$2^{-\ell/2}H/h$ & $\Delta_2$&$2^{-\ell/2}H/h$&$\Delta_2$&$2^{-\ell/2}H/h$\tabularnewline
				\midrule 
				10&1/40& 1/10 &  0.359514 & 4.0000   & 0.319254 & 2.8284 &0.311940 & 2.0000\tabularnewline	\midrule 
				20&1/100& 1/20&  0.362469 & 5.0000   &0.346776 & 3.5355 &0.334599 & 2.5000\tabularnewline\midrule
				40&1/280& 1/40&  0.367140 & 7.0000  &0.357546 & 4.9497 &0.346937 & 3.5000\tabularnewline\midrule
				80&1/720& 1/80&  0.368978 & 9.0000  &0.360842 & 6.3640 & 0.351649 & 4.5000\tabularnewline\midrule
				160&1/2080&1/160 & 0.370396 & 13.0000   & 0.365802 & 9.1924 &0.357377 & 6.5000\tabularnewline\midrule
				320&1/5760& 1/320& 0.370988& 18.0000   &  0.368373 & 12.7279 &0.361781 & 9.0000\tabularnewline
				\toprule 				
				\multicolumn{9}{c}{$h= 1/(3k)$, $H=2/k$}\tabularnewline
				\midrule 
				\multirow{2}{*}{$k$}&  \multirow{2}{*}{$h$}&  \multirow{2}{*}{$H$}& \multicolumn{2}{c}{$\ell=0$} & \multicolumn{2}{c}{$\ell=1$} & \multicolumn{2}{c}{$\ell=2$}\tabularnewline
				\cmidrule{4-9} \cmidrule{5-9} \cmidrule{6-9} \cmidrule{7-9} \cmidrule{8-9} \cmidrule{9-9}
				&&	&$\Delta_2$&$2^{-\ell/2}H/h$ & $\Delta_2$&$2^{-\ell/2}H/h$&$\Delta_2$&$2^{-\ell/2}H/h$\tabularnewline
				\midrule
				10&1/30& 1/5 & 0.37018 & 6.0000   & 0.35174 & 4.2426  & 0.34388 & 3.0000 \tabularnewline	\midrule 
				20&1/60& 1/10& 0.37018 & 6.0000 & 0.35174 & 4.2426  & 0.34388 & 3.0000\tabularnewline\midrule
				40&1/120& 1/20&0.37018 & 6.0000  & 0.35174 & 4.2426  & 0.34388 & 3.0000\tabularnewline\midrule
				80&1/240& 1/40&0.37018 & 6.0000 & 0.35174 & 4.2426  & 0.34388 & 3.0000\tabularnewline\midrule
				160&1/480&1/80 &0.37018 & 6.0000  & 0.35174 & 4.2426  & 0.34388 & 3.0000\tabularnewline\midrule
				320&1/960& 1/160&0.37018 & 6.0000 &0.35174 & 4.2426  & 0.34388 & 3.0000\tabularnewline
				\bottomrule				
			\end{tabular}
			
		\caption{Comparison of computed value $\Delta_2$ with $2^{-\ell/2}H/h$. \label{ta:test2b}}
	\end{table}

\subsection{Test 3: EMs-HS for homogeneous cases} 
Next, we show the performance of EMs-HS in terms of the iteration numbers with biconjugate gradient stabilized (BiCGSTAB) method, implemented in MATLAB. In this test, we take the zero vector as the initial guess and a reduction of the relative residual by $10^{-8}$ as the stopping criterion. We present in Tables \ref{ta:test4a} and \ref{ta:test4b} the iteration numbers of EMs-HS for different levels and wavenumbers in 2-d and 3-d examples, respectively. The iteration number depends only weakly on $k$ in {\em Setting A}, and EMs-HS is robust with respect to $k$ in 2-d case when $\ell>0$ in {\em Setting B}.

\begin{table}[htb!]
		\centering \begin{tabular}{ccccccccc}
			\toprule 
			\multicolumn{6}{c}{$h\simeq k^{-3/2}$, $H=1/k$}\tabularnewline
			\midrule 
			
			$k$& $h$ & $H$ &   {$\ell=0$}&   {$\ell=1$}&  {$\ell=2$}\tabularnewline\hline
			10 &1/40&1/10& 10&7&3   \tabularnewline\hline
			20&$1/100$&1/20& 12&9&5    \tabularnewline\hline
			40&$1/280$&1/40&16 &10&7  \tabularnewline\hline
			80&$1/720$&1/80&  22&11&7    \tabularnewline\hline
			160&1/2080 &1/160& 32 &15&8   \tabularnewline\hline
			320&1/5760&1/320& 54 &18&10   \tabularnewline
            \bottomrule 
           \end{tabular}
           \begin{tabular}{ccccccccc}
			\toprule 
			\multicolumn{6}{c}{$h= 1/(3k)$, $H=2/k$}\tabularnewline
			\midrule 
			
			$k$& $h$ & $H$ &  {$\ell=0$}&   {$\ell=1$}&  {$\ell=2$}\tabularnewline\hline
			10&1/30& 1/5& 16&8&7   \tabularnewline\hline
			20&1/60& 1/10& 17&8&7    \tabularnewline\hline
			40&1/120& 1/20& 21&8&7  \tabularnewline\hline
			80&1/240& 1/40&33 &8&7    \tabularnewline\hline
			160&1/480&1/80&69 &8&7   \tabularnewline\hline
			320&1/960& 1/160&206 &8&7  \tabularnewline\bottomrule		
		\end{tabular}
		\caption{Iteration numbers of EMs-HS under different level parameters $\ell$ and wavenumbers $k$, 2-d tests.}
		\label{ta:test4a}
	\end{table}
	
	\begin{table}[hbt!]
		\centering \begin{tabular}{cccc}
			\toprule 
			\multicolumn{4}{c}{$h\simeq k^{-3/2}$, $H=1/k$}\tabularnewline
			\midrule 			
			$k$& $h$ & $H$ &  $\ell=0$\tabularnewline\hline
			10&$1/40$&1/10&  11    \tabularnewline\hline
			20&$1/100$&1/20& 13   \tabularnewline
            \bottomrule
            \end{tabular} 
 \quad
   \begin{tabular}{cccc}
   \toprule
			\multicolumn{4}{c}{$h= 1/(3k)$, $H=2/k$}\tabularnewline
			\midrule 	
			$k$& $h$ & $H$ &   $\ell=0$\tabularnewline\hline
			10&$1/30$&1/5& 16     \tabularnewline\hline
			20&$1/60$&1/10& 18  \tabularnewline\hline
			40&$1/120$&1/20& 24 \tabularnewline\bottomrule	
		\end{tabular}
		\caption{Iteration numbers of the EMs-HS under level parameter $\ell:=0$ and different wavenumbers $k$, 3-d tests. }
		\label{ta:test4b}
	\end{table}

 \subsection{Test 4: EMs-HS for heterogeneous cases}
Finally, we test the performance of EMs-HS using three strongly heterogeneous models, which are much more challenging. The first model has a homogeneous background but there are perforations (white region in Figure \ref{fig:m1})  inside the computational domain. The second model (Figure \ref{fig:m2}) and third model (Figure \ref{fig:m3}) are two heterogeneous models whose background velocity is 1 m/s and inside which fractures and inclusions with a velocity 1.4 m/s lie.
We consider three preconditioners: EMs-HS preconditioner, the polynomial coarse space based hybrid preconditioner and the one-level preconditioner $B_{\op{RAS-imp}}$. The iteration numbers and 
the CPU time of the BiCGSTAB solve stage with these three preconditioners are shown in Table \ref{ta:test4}, where we fix the fine-scale and coarse space size and vary the wavenumber. 

The EMs-HS preconditioner outperforms polynomial based hybrid preconditioner especially for the two 2-d models. Moreover, the polynomial based hybrid preconditioner fails to generate a solution in most 2-d tests. Even when it converges, it requires much more iterations than the one-level preconditioner $B_{\op{RAS-imp}}$. This is attributed to the fact that the polynomials fail to resolve the heterogeneous velocity and source frequency. In contrast, the EMs-HS preconditioner converges in all cases since $V_{h,0}$ transfers global information on the heterogeneous velocity and wavenumber to the local solutions, and can boost the performance of the one-level preconditioner $B_{\op{RAS-imp}}$. Hence it is overall more efficient than the one-level preconditioner $B_{\op{RAS-imp}}$. The real part of the solutions for these three models are displayed in Figures \ref{fig:m1}-\ref{fig:m3}, which illustrate the 
complicated wavefields due to heterogeneity in the velocity fields.

   \begin{table}[hbt!]
	\centering
	\begin{adjustbox}{max width=\textwidth}
		\begin{tabular}{ccccccccccc}\toprule
			\multicolumn{7}{c}{Homogeneous model with perforation, $h=1/1000$, $H=1/50$}\tabularnewline\hline
			\multirow{2}{*}{ $k$} & \multicolumn{2}{c}{EMs-HS: $\ell=1$} &\multicolumn{2}{c}{$Q_2$ based $B_{\op{hybrid}}$} &  \multicolumn{2}{c}{$B_{\op{RAS-imp}}$} \tabularnewline
			\cline{2-7} 
			&Iter  & Tsol  &    Iter  & Tsol& Iter  & Tsol \tabularnewline\hline
			157.1	&13&3.0&/ &/ &209&29.9 \tabularnewline\hline 
			179.5	&12&3.0&/ & /&192&27.3 \tabularnewline\hline 
			209.4&17&4.2&/ & /&223&31.6 \tabularnewline
			\midrule
			\multicolumn{7}{c}{2-d heterogeneous model, $h=1/1000$, $H=1/50$}\tabularnewline\hline
			\multirow{2}{*}{ $k$} & \multicolumn{2}{c}{EMs-HS: $\ell=1$} &\multicolumn{2}{c}{$Q_2$ based $B_{\op{hybrid}}$} &  \multicolumn{2}{c}{$B_{\op{RAS-imp}}$} \tabularnewline
			\cline{2-7} 
			&Iter  & Tsol  &    Iter  & Tsol& Iter  & Tsol \tabularnewline\hline
			$128.2\sim 179.5$&13&3.5& 827& 173.9&279&43.9 \tabularnewline\hline 
			$149.6\sim 209.4$	&20&5.7& 847&178.2 &251&38.7\tabularnewline\hline 
			$179.5\sim 251.3$	&51&12.2& /&/&273&43.6\tabularnewline
			\midrule
			\multicolumn{7}{c}{3-d heterogeneous model, $h=1/200$, $H=1/20$}\tabularnewline\hline
			\multirow{2}{*}{ $k$} & \multicolumn{2}{c}{EMs-HS: $\ell=0$} &\multicolumn{2}{c}{$Q_1$ based $B_{\op{hybrid}}$} &  \multicolumn{2}{c}{$B_{\op{RAS-imp}}$} \tabularnewline
			\cline{2-7} 
			&Iter  & Tsol  &    Iter  & Tsol& Iter  & Tsol \tabularnewline\hline
			$29.9\sim 41.9$		&12&551.3 &18&639.5 &113 &3529.7 \tabularnewline\hline 
			$35.9\sim 50.3 $&21&964.3 &37&1279.4  &106&3494.3 \tabularnewline\bottomrule 
			
		\end{tabular}
	\end{adjustbox}
	\caption{ Iteration numbers ("Iter") and CPU time ("Tsol") in seconds of the BiCGSTAB solve stage with EMs-HS preconditioner, the polynomial coarse space based hybrid Schwarz preconditioner and the one-level preconditioner $B_{\op{RAS-imp}}$. "/" indicates that the preconditioner exits the BiCGSTAB iteration before reaching the convergence criteria.}
 \label{ta:test4}
\end{table}

\begin{figure}[hbt!]
	\centering
	\includegraphics[trim={0cm .5cm 0cm .5cm},clip,width=3in]{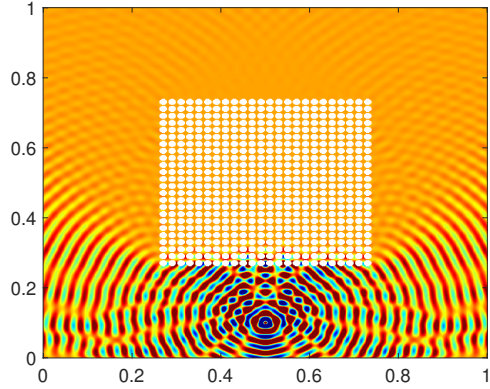}
	\caption{The solution of the homogeneous model with perforation, $k=209.4$.}
\label{fig:m1}
\end{figure}

\begin{figure}[hbt!]
	\centering
	\subfigure[A 2-d heterogeneous model, velocity in blue region is 1 m/s and red region is 1.4 m/s. ]{
		\includegraphics[trim={0cm .5cm 0.00cm 0cm},clip,width=3.0in]{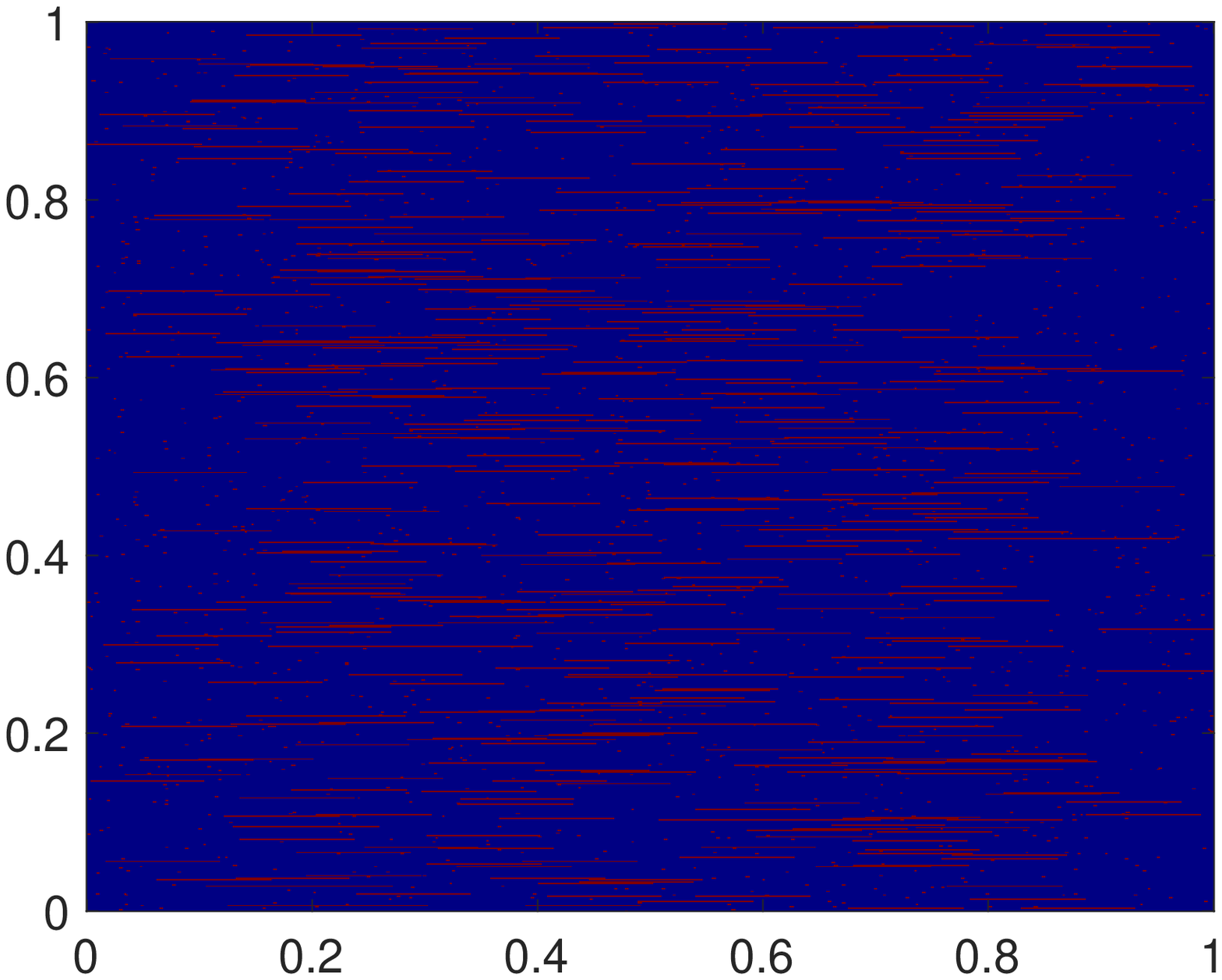}}
	\subfigure[Solution of the 2-d heterogeneous model, $k=179.5\sim 251.3$]{
		\includegraphics[trim={0cm .5cm 0.00cm 0cm},clip,width=3.0in]{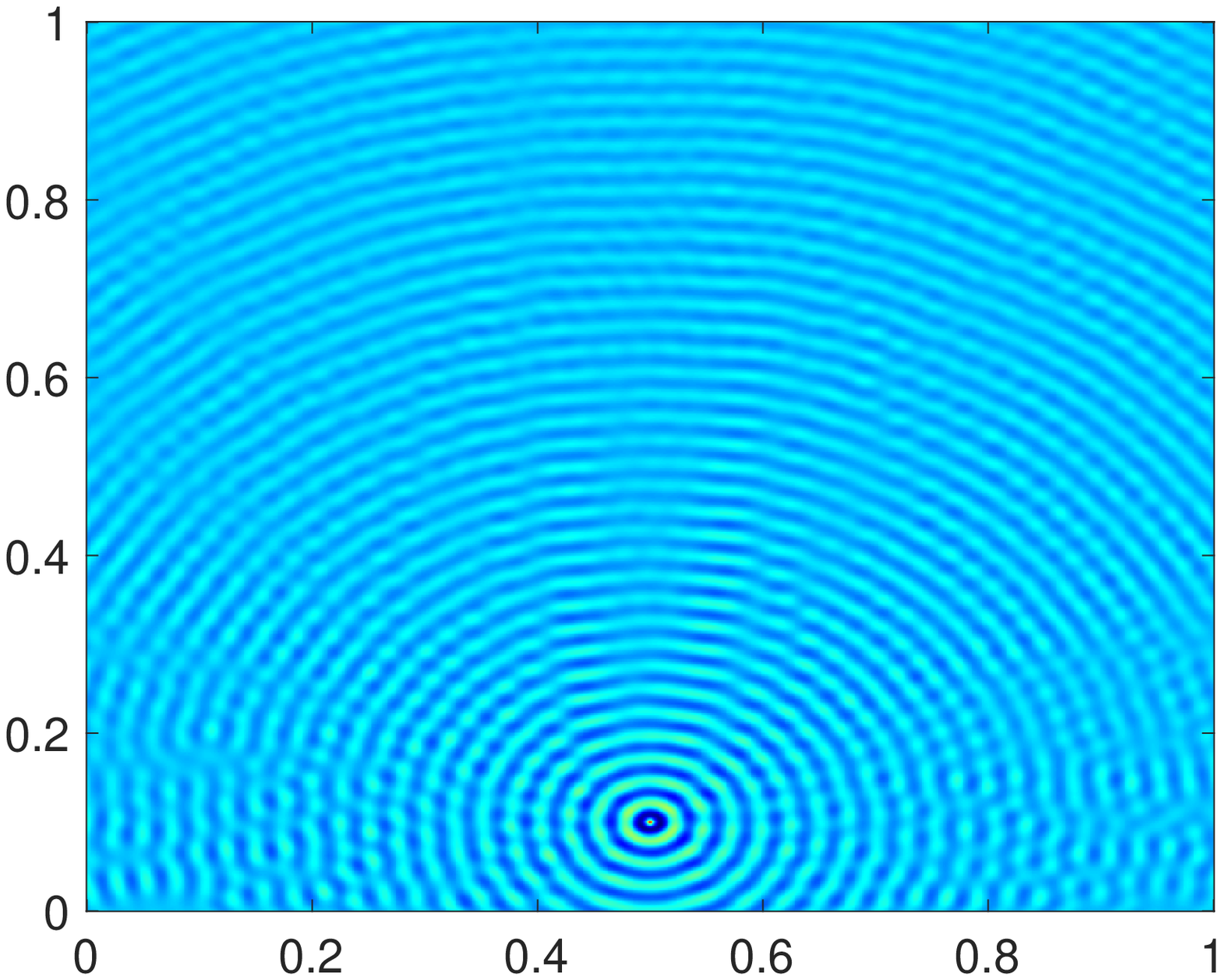}}
				
    \caption{The velocity field and solution of the 2-d heterogeneous model.}\label{fig:m2}
\end{figure}

	\begin{figure}[hbt!]
	\centering
	\subfigure[A 3-d heterogeneous model, velocity in blue region is 1 m/s and red region is 1.4 m/s.]{
		\includegraphics[trim={0cm .5cm 0.00cm 0cm},clip,width=2.9in]{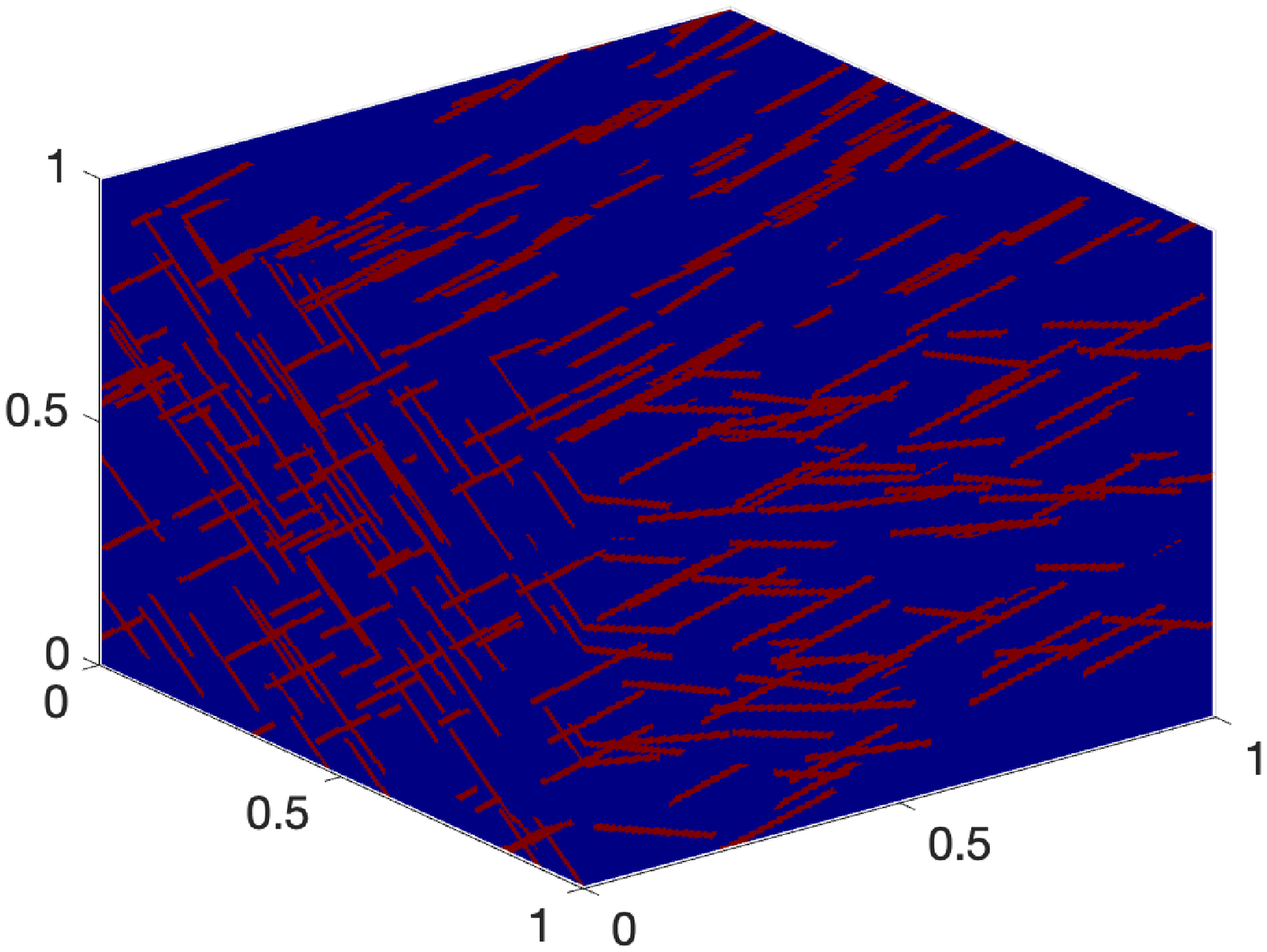}}	
	\subfigure[Solution of the 3-d heterogeneous model, $k=35.9\sim 50.3 $]{
		\includegraphics[trim={0cm .5cm 0.00cm 0cm},clip,width=3.0in]{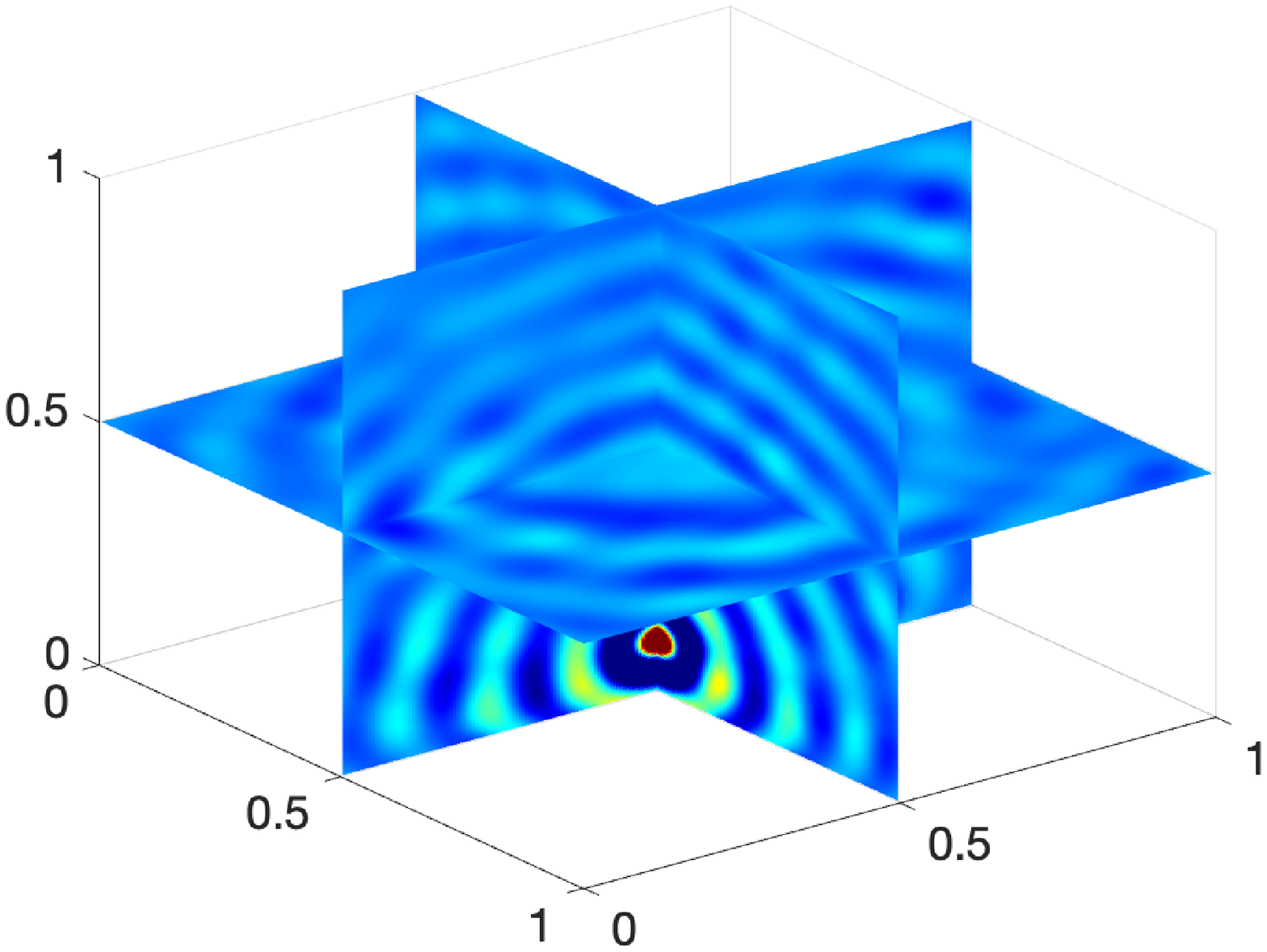}}				
		
      \caption{The velocity field and solution of the 3-d heterogeneous model.}
      \label{fig:m3}
\end{figure}
  
 \section{Conclusion}\label{sec:conclusion}
 We have designed a novel hybrid Schwarz preconditioner for the Helmholtz problem with large wavenumbers based on the edge multiscale ansatz space, which allows using overlaps of size of fine grid. We have rigorously analyzed its  convergence, and demonstrated the performance of the method with extensive numerical tests in both 2-d and 3-d. There are several interesting questions deserving further investigations. Is it possible to design scalable solvers for the coarse problems, of which the dimension is increasing with respective to the wavenumber $k$?  Whether we can design a similar hybrid Schwarz preconditioner for the time-harmonic Maxwell equation with large wavenumber?
\bibliographystyle{abbrv}
\bibliography{reference}
\end{document}